\documentclass[12pt]{article}

\usepackage{amsmath}
\usepackage{amssymb}
\usepackage{latexsym}

\usepackage{xcolor}
\usepackage{graphicx}

\newcommand{\C}{\mathbb{C}}

\newcommand{\R}{\mathbb{R}}

\renewcommand{\phi}{\varphi}

\renewcommand{\leq}{\leqslant}
\renewcommand{\geq}{\geqslant}

\newcommand{\PB}[1]{\raisebox{-10pt}[10pt][0pt]{$\text{\Huge{*}}$}{}_{#1}}

\newcommand{\st}{\ \vert\ }

\renewcommand{\Bar}[1]{\overline{#1}}
\renewcommand{\Tilde}[1]{\widetilde{#1}}

\newcommand{\co}{\colon\thinspace}

\usepackage{amsthm}

\newtheorem{thm}{Theorem}[section]
\newtheorem{prop}[thm]{Proposition}
\newtheorem{lem}[thm]{Lemma}
\newtheorem{cor}[thm]{Corollary}

\theoremstyle{definition}
\newtheorem{df}[thm]{Definition}
\newtheorem{ej}[thm]{Example}

\theoremstyle{remark}
\newtheorem{rem}[thm]{Remark}
\newtheorem{rems}[thm]{Remarks}

\usepackage[all]{xy}

\newdir{ >}{{}*!/-5pt/@{>}} 

\usepackage{mathrsfs}
\newcommand{\mscr}[1]{\mathscr{#1}}

\newcommand{\dg}{\mscr{D}\mscr{F}} 
\newcommand{\dgv}[1]{\mscr{D}\mscr{F}^{.#1}} 
\newcommand{\wg}{\mscr{W}} 

\usepackage{array}

\usepackage{subfig}
\usepackage{caption}
\usepackage{xspace}

\newcommand{\nvb}[1]{{$#1$}-fold vector bundle\xspace}
\newcommand{\nvbs}[1]{{$#1$}-fold vector bundles\xspace}

\newcommand{\dvb}[8]{\xymatrix{#1 \ar[r]^{#6} \ar[d]_{#5} & #3 \ar[d]^{#8} \\ #2 \ar[r]^{#7} & #4}}
\newcommand{\dvbs}[4]{\dvb{#1}{#2}{#3}{#4}{}{}{}{}}

\newcommand{\decom}[1]{\overline{#1}}

\newcommand{\duer}{\mathbin{\ensuremath{\mathaccent\times\vert}}} 

\newcommand{\hop}[1]{\underline{\ensuremath{#1}}} 
\newcommand{\ph}[1]{p(#1)} 
\newcommand{\run}[1]{r(#1)} 

\newcommand{\BE}{E_\bullet} 

\DeclareMathOperator{\Dec}{Dec}

\newcommand{\parti}[1]{\mathscr{P}_{#1}}
\newcommand{\V}[1]{V_{#1}}
\newcommand{\G}[1]{\mathscr{L}_{#1}}

\newcommand{\stsec}[2]{\genfrac{\{}{\}}{0pt}{}{#1}{#2}} 

\newcommand{\id}{\operatorname{id}}

\newcommand{\comp}{\star}  

\newcommand{\pair}[2]{\left\langle #1\,\vert\,#2\right\rangle}

\newcommand{\ed}[2]{c_{#1\,#2}}
\newcommand{\gra}[1]{\overline{K}_{#1}}

\usepackage{lscape}

\usepackage{url}

\begin{document}

\title{\textbf{Duality functors for $n$-fold vector bundles}\footnote{Mathematics Subject 
Classification (MSC2000): 53D17 (primary), 18D05, 18D35, 20E99, 55R99 (secondary).}
\footnote{Keywords:\ double vector bundles, triple vector bundles, $n$-fold vector bundles, 
multiple vector bundles, duality, extensions of symmetric groups.}} 

\author{
\begin{tabular}{lcl}
Alfonso Gracia-Saz &  & K. C. H. Mackenzie\\
Department of Mathematics \& Statistics  & &   School of Mathematics \& Statistics\\
University of Victoria                    & &  University of Sheffield\\
PO BOX 3060 STN CSC                       & &  Sheffield, S3 7RH\\
Victoria, B.C.                            & &  United Kingdom\\
Canada V8W 3R4                            & &  \\
\url{alfonso@uvic.ca} & &   \url{K.Mackenzie@sheffield.ac.uk}
\end{tabular}
}

\date{\today}

\maketitle

\begin{abstract}
Double vector bundles may be dualized in two distinct ways and these duals are themselves 
dual. These two dualizations generate a group, denoted $\dg_2$, which is the symmetric group 
$S_3$ on three symbols. In the case of triple vector bundles the authors proved in a 
previous paper
that the corresponding group $\dg_3$ is an extension of $S_4$ by the Klein four-group. In 
this paper we show that the group $\dg_n$, for $n$-fold vector bundles, $n\geq 3$, is an 
extension of $S_{n+1}$ by a certain product of groups of order 2, and show that the centre
is nontrivial if and only if $n$ is a multiple of $4$. The methods employ an interpretation
of duality operations in terms of certain graphs on $(n+1)$ vertices. 
\end{abstract}


\section{Introduction}
\label{sect:int}

In a previous paper \cite{Gracia-SazM:2009}, the authors showed that the
group of duality functors of triple vector bundles has order 96, and is
an extension of the symmetric group $S_4$ by the Klein four-group. This 
followed work by one of us on the duality of double vector bundles 
\cite{Mackenzie:1999,Mackenzie:2005dts}.

Duality for double and multiple vector bundles originated in Poisson geometry; 
see \cite{Pradines:1988,MackenzieX:1994} and references there. 
In \cite{Mackenzie:1999} one of us applied Pradines' duality for vector bundle 
objects in the category of groupoids \cite{Pradines:1988} to double vector 
bundles and showed that the two duals of a double vector bundle are themselves dual. 
This result was so unexpected that it was natural to investigate the triple and higher 
cases. 

Double vector bundles have also been used in treatments of connection theory and of 
theoretical mechanics for many years 
\cite{Dieudonne:III,Besse:MAWGC,Tulczyjew:1977}, though without any consideration 
of duality. Voronov \cite[\S6]{Voronov:QmMt} has begun the study of bracket
structures on multiple vector bundles, using super techniques. 
In the present paper we are only concerned with duality for unstructured
multiple vector bundles, and we do not consider bracket structures or 
geometric applications. 

So far as we know, these groups have not appeared before; in particular, they
do not seem to be a modern formulation of a classical construction. 
Their significance is not fully clear, but the relation (\ref{eq:xyxyxy}) below 
which defines the duality group for $n = 2$ is the key to the compatibility of 
Lie algebroid structures on a double vector bundle \cite{Mackenzie:2011}. Whether 
there are corresponding results for bracket structures on multiple vector bundles 
for other $n$ will be investigated elsewhere. For the moment, we only remark that 
the sequence of groups for $n=2,3,4$ shows new features at each term and this is 
sufficient reason to investigate whether the sequence becomes regular.

Before describing the main results of the paper, we recall the results 
of the double case from \cite{Gracia-SazM:2009}. 
For double and perhaps triple vector bundles, the notation used is cumbersome,  
but it is designed to handle the $n$-fold case, which is the main concern of the paper. 

A \emph{double vector bundle} is a manifold $E_{1,2}$ with two vector bundle structures, over 
bases $E_1$ and $E_2$, each of which is a vector bundle on a manifold $M$, such that the 
structure maps of $E_{1,2}\to E_1$ (the bundle projection, the addition, the scalar 
multiplication, the zero section) are morphisms of vector bundles with respect to the other 
structure. We write $E$ to denote the entire structure, and sometimes for the total space. 
See Figure~\ref{fig:dvbs}(a),
\begin{figure}[h]
\begin{center}
\subfloat[]%
{\xymatrix@=8mm{
E_{1,2} \ar[r] \ar[d] & E_2 \ar[d]  &\\
E_1    \ar[r]   & M \\
}}
\qquad
\subfloat[]%
{\xymatrix@=8mm{
E_{1,2}\duer {E_1} \ar[r] \ar[d] & E_{12}^* \ar[d]  &\\
E_1    \ar[r]   & M \\
}}
\qquad
\subfloat[]%
{\xymatrix@=8mm{
E_{1,2}\duer{E_2} \ar[r] \ar[d] & E_2 \ar[d]  &\\
E_{12}^*    \ar[r]   & M \\
}}
\end{center}
\caption{(a) shows a double vector bundle and (b) and (c) its vertical and 
horizontal duals.\label{fig:dvbs}}
\end{figure}
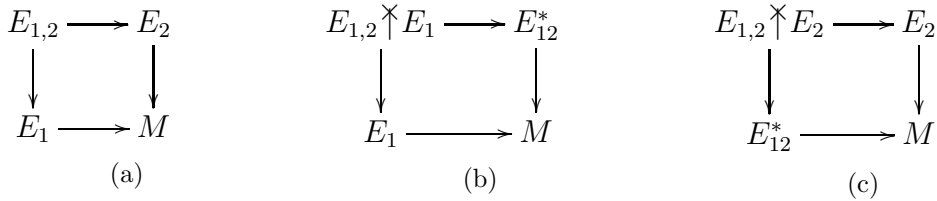

Concisely, a double vector bundle is a vector bundle object in the category of vector bundles. 
A definition in full detail is given in \cite[Chap.~9]{Mackenzie:GT}. 

The \emph{core} \cite{Pradines:DVB} of a double vector bundle $E$ is the set of 
elements $e\in E_{1,2}$ which project to zero in both $E_1$ and $E_2$. The core, 
which we denote $E_{12}$ (without the comma), is closed under both additions in $E$ 
and these additions coincide, giving $E_{12}$ a natural structure of vector bundle over $M$. 

Dualizing $E_{1,2}$ over its base $E_1$ leads to a double vector bundle as in 
Figure~\ref{fig:dvbs}(b), in which $E_2$ has been `replaced' by the dual of the core. 
We write the base over which the dualization takes place on the line to avoid multiple
superscripts, and use the symbol 
$\duer$ to avoid confusion with other uses of stars and asterisks. 

Likewise dualizing $E_{1,2}$ over $E_2$ leads to the double vector bundle shown in 
Figure~\ref{fig:dvbs}(c). For brevity, denote the action of dualizing vertically and
horizontally by $E^X$ and $E^Y$. 
Repeating these operations leads to the double vector bundles shown in Figure~\ref{fig:all6}. 

\begin{figure}[h]
\begin{center}
\subfloat
{\xymatrix@=8mm{
E \ar[r] \ar[d] & E_2 \ar[d]  &\\
E_1    \ar[r]   & M \\
}}
\qquad
\subfloat
{\xymatrix@=8mm{
E^X \ar[r] \ar[d] & E_{12}^* \ar[d]  &\\
E_1    \ar[r]   & M \\
}}
\qquad
\subfloat
{\xymatrix@=8mm{
E^{XY} \ar[r] \ar[d] & E_{12}^* \ar[d]  &\\
E_2    \ar[r]   & M \\
}}
\\
\subfloat
{\xymatrix@=8mm{
E^{XYX} \ar[r] \ar[d] & E_1 \ar[d]  &\\
E_2    \ar[r]   & M \\
}}
\qquad
\subfloat
{\xymatrix@=8mm{
E^Y \ar[r] \ar[d] & E_2 \ar[d]  &\\
E_{12}^*    \ar[r]   & M \\
}}
\qquad
\subfloat
{\xymatrix@=8mm{
E^{YX} \ar[r] \ar[d] & E_1 \ar[d]  &\\
E_{12}^*    \ar[r]   & M \\
}}
\end{center}
\caption{\label{fig:all6}}
\end{figure}

In particular $E^{XYX}$ has side bundles $E_1$ and $E_2$ but these have been 
interchanged. Note that $E^{XYX}$ is not canonically isomorphic to the \emph{flip} 
of $E$, by which we mean the double vector bundle obtained by interchanging the 
two structures on $E$, as for the canonical involution on a double tangent bundle. 
Roughly speaking, although $E^{XYX}$ has side bundles $E_1$ and $E_2$, its core 
has been reversed; see the discussion following Prop.~2.10 in \cite{Gracia-SazM:2009}. 
However $E^{YXY}$ is canonically isomorphic to $E^{XYX}$ and we therefore write
\begin{equation}
\label{eq:xyxyxy}
XYX = YXY. 
\end{equation}
Together with $X^2 = Y^2 = 1$ it follows that 
the dualizations of a double vector bundle form the symmetric group on three symbols, 
and every sequence of dualization operations applied to $E$ results in a double
vector bundle canonically isomorphic to one of those in Figure~\ref{fig:all6}. 

As \cite{Gracia-SazM:2009} showed, it is crucial to regard the operations $X$ and $Y$ as 
functors on appropriate categories. We denote the group of these \emph{dualization
functors} by $\dg_2$. 

Write $E_0 = E_{12}^*$ so that the two side bundles and the dual of the core of 
$E$ are $E_1, E_2, E_0$. The core of $E^X$ is $E_2^*$ and so the side bundles and the 
dual of the core of $E^X$ are $E_1, E_0, E_2$. The side bundles and the dual of the core 
of $E^Y$ are $E_0, E_2, E_1$. Thus $\dg_2$ can be regarded as the symmetric group on the 
bundles $E_1, E_2, E_0$. 

The notation for a triple vector bundle, as used in \cite{Gracia-SazM:2009}, is shown 
in Figure~\ref{fig:triple}(a). The total space is denoted by $E_{1,2,3}$, the double
vector bundles which form the lower faces are denoted by $E_{i,j}$, and we write $E_i$ for 
the vector bundles which form the edges abutting $M$. 
In all figures we read oblique arrows as coming out of the page.

\begin{figure}[h]
\subfloat[A triple vector bundle \dots]%
{\xymatrix@=3mm{
E_{1,2,3}\ar[rr]^{X}\ar[dd]^{Z}\ar[rd]^{Y}   & & E_{2,3}\ar'[d][dd] \ar[dr] &\\
& E_{1,3}   \ar[rr] \ar[dd]    & & E_3 \ar[dd]  \\
E_{1,2}\ar'[r][rr] \ar[dr]  &       & E_2 \ar[dr] &\\
& E_1  \ar[rr] & & M\\
}}
\quad
\subfloat[\dots{} and its core structure.]%
{\xymatrix@=1mm{
&& E_{3,12}\ar[dddd] &&\\
& E_{2,31}\ar[rrdd] &&&\\
E_{1,23}\ar[rrrr] &&&& E_{23}\\
&&& E_{13} &\\
&& E_{12} &&\\
}}
\quad
\subfloat[The dual of $E$ over $E_{2,3}$]%
{\xymatrix@=2mm{
{E_{1,2,3} \duer E_{2,3}} \ar[rr] \ar[dd] \ar[dr] & & E_{2,3} \ar'[d][dd] \ar[dr] &\\
& E_{3,12}\duer E_3 \ar[rr] \ar[dd]         & & E_3 \ar[dd]  \\
E_{2,31}\duer E_2  \ar'[r][rr] \ar[dr]         & & E_2 \ar[dr] &\\
& E_{123}^*  \ar[rr] & & M\\
}}
\caption{\label{fig:triple}}
\end{figure}

Figure~\ref{fig:triple}(b) shows the core structure of $E$. The core of $E_{i,j}$
is denoted $E_{ij}$ without the comma; it is a 
vector bundle over $M$. The core of the top face of $E$ is denoted $E_{3,12}$; 
it is a vector bundle over $E_3$. The vector bundle structure of $E_{1,2,3}$ with 
base $E_{1,2}$ restricts to $E_{3,12}$ and gives it also the structure of a vector 
bundle over $E_{12}$; with these two structures $E_{3,12}$ is a double vector bundle 
with side bundles $E_3$ and $E_{12}$. 

The cores of the rear and left faces are denoted $E_{2,31}$ and $E_{1,23}$ and 
are double vector bundles in a similar way. 

These three \emph{core double vector bundles} have the same core, called the
\emph{ultracore}; this is denoted $E_{123}$ and is a vector bundle over $M$. It is 
the set of all elements of $E_{1,2,3}$ which project to the double zero element in 
each of $E_{1,2}$, $E_{2,3}$ and $E_{3,1}$. 

For a triple vector bundle one may dualize in three directions, which we denote for the
moment by $X$, $Y$ and $Z$. Take the dual over $E_{2,3}$, as in Figure~\ref{fig:triple}(c),
to be $E^X$ and write
$E_0 = E_{123}^*$. Then $X$ exchanges $E_1$ and $E_0$ and leaves $E_2$ and $E_3$
fixed. In terms of the effect which $X$ has on the four bundles $E_1$, $E_2$, $E_3$
and $E_0$, we can regard $X$ as the transposition $(0\,1)$ in $S_4$. Likewise $Y$
acts as $(0\,2)$ and $Z$ as $(0\,3)$. These three transpositions generate $S_4$, so 
for every $\sigma\in S_4$ there is a word $W$ in $X$, $Y$ and $Z$ which acts as~$\sigma$. 

To express this more precisely, write $\dgv{1}_3$ for the group on $X, Y, Z$
subject to $X^2 = Y^2 = Z^2 = 1$ and to
\begin{equation}
\label{eq:xyz}
(XY)^3 = (YZ)^3 = (ZX)^3 = 1.   
\end{equation}
The group $\dgv{1}_3$ should be thought of as `version~1' of the duality functor
group for triple vector bundles. The discussion above shows that there is a surjective 
morphism from $\dgv{1}_3$ to the symmetric group $S_4$. 

In particular, the word $(XYXZ)^2$ is mapped to the identity of $S_4$. 
It is shown in \cite{Gracia-SazM:2009} that $(XYXZ)^2$ has order 2 in $\dgv{1}_3$, 
and that the surjection $\dgv{1}_3\to S_4$ has for kernel a
Klein four--group $K_4$, consisting of $(XYXZ)^2$ together with its conjugates
and the identity. 

Now define $\dgv{2}_3$ to be the quotient of $\dgv{1}_3$ over the relations
\begin{equation}
\label{eq:xyxz}
(XYXZ)^4 = (YZYX)^4  = (ZXZY)^4 = 1. 
\end{equation}
Then the above discussion may be formulated as the statement that $\dgv{2}_3$ is the 
duality functor group $\dg_3$. 

This completes a brief review of the main results in the triple case. In the
present paper we first show that, for all $n\geq 4$, there is again a short exact sequence
\begin{equation}  
\label{eq:ses-intro}
\xymatrix{1 \ar[r] & K_{n+1} \ar[r] & \dg_n \ar[r]  & S_{n+1} \ar[r] & 1.}
\end{equation}
where $K_{n+1}$ is a direct product of copies of $C_2$, the cyclic group of order 2. 
To calculate the number of copies we introduce in \S\ref{sect:tgk} an equivalent 
description of $K_{n+1}$ in terms of certain graphs on $n+1$ vertices. This interpretation 
was not needed in the case $n = 3$ but is the key to the cases $n\geq 4$. We find in
Corollary \ref{cor:orderK} that $K_{n+1}$ is the direct product of $\frac{1}{2}(n+1)(n-2)$ 
copies of $C_2$. 

The $\dg_n$ are thus unexpectedly large and, as far as we know, have no precedent in
earlier or classical work. It is tempting to look for smaller groups which embody equivalent
information, but we see no prospect of this. The relationship between $\dg_n$ and 
$K_{n+1}$ is reminiscent of the relationship between braids and pure braids, or between 
gauge transformations and pure gauges. In these theories, it is generally sufficient to 
concentrate on the pure case, but that is not so here; one cannot focus just on $K_{n+1}$.  
This is already clear from the double and triple cases. 

We show in \S\ref{sect:description}, using the description of the kernel in terms of 
graphs, that (\ref{eq:ses-intro}) splits for $n = 4$ and for $n= 2 \pmod{4}$; we do 
not know what the situation is for $n = 8$. We also show that for $n= 0 \pmod{4}$ the 
centre of $\dg_n$ has order 2, and is trivial for all other~$n$. 

A complete description of $\dg_n$ for general $n$ will probably require a set of relations
in terms of the dualization operators. For double and triple vector bundles the relations
are as given in (\ref{eq:xyz}) and (\ref{eq:xyxz}). 
At the end of the paper we find that the relations for $\dg_4$, in addition to those 
corresponding to (\ref{eq:xyz}) and (\ref{eq:xyxz}), require words of length 24 and 32. 
It will be interesting if these have geometric interpretations like that of `cornering' 
for the relation in the double case \cite{Mackenzie:2005dts}. 

The paper divides into two parts. In the first, consisting of Section \ref{sect:not}
and Section \ref{sect:dfg}, we are concerned to set up the notation and terminology 
needed to work with $n$-fold vector bundles. 
This takes some time, but it is necessary
to have a systematic notation for the spaces associated with dualizing an $n$-fold 
vector bundle before introducing the notation for the appropriate concept of
automorphism. As we emphasized in \cite{Gracia-SazM:2009}, the problem --- when are 
the results of two sequences of dualization operations canonically isomorphic? --- is 
itself difficult to formulate effectively, and that is what makes the
material of Section \ref{sect:not} and Section \ref{sect:dfg} necessary. 

The second part concerns the actual calculations for the duality groups. Many readers 
may prefer to go directly to Section \ref{sec:thetaX} and refer back as needed. 
As much as possible, we have avoided repeating material from \cite{Gracia-SazM:2009}. 

\section{$n$-fold vector bundles}
\label{sect:not}

In order to work with $n$-fold vector bundles, we need an effective notation
for the various side bundles and cores. The actual definition of an $n$-fold 
vector bundle is in \ref{df:nvb}. 

We first extend the notation used in Figures~\ref{fig:dvbs} and \ref{fig:triple}. 
The top space of an $n$-fold vector bundle will generally be denoted by $E_{1,\dots,n}$. 
The various faces will be $k$-fold vector bunndles for $k\leq n$, indexed by $k$-element 
subsets of $\{1, \dots, n\}$ and the cores of the various faces will be denoted by removing
appropriate commas from the suffices. Thus we need a notation which makes clear what 
is denoted by, for example, $E_{1,234,5}$. Further, when considering duals we will
need to start with index sets other than $\{1, \dots, n\}$. 

Later in the paper we will need to use other sets of sets of integers to denote 
certain maps of $n$-fold vector bundles, and for this reason we use a distinctive
terminology, `hops and runs', for the sets which index the faces and cores. 

\subsection*{Hops and runs} 
\addcontentsline{toc}{subsection}{\protect\numberline{}{Hops and runs}}

Let $A$ be a finite set.  
A \emph{hop} across $A$ is a set of non-empty, disjoint subsets of $A$.   
The \emph{length} of a hop is the number of subsets it contains.

For instance, if $A = \{ 1, 2, 3, 4, 5\}$, then $\{ \{1, 3\}, \{4\} \}$ is a hop across 
$A$ of length 2. When there is no ambiguity, we may omit the curly braces; in that 
case, we will separate different subsets by commas, and we will not use a comma to 
separate elements in the same subset.  For instance, instead of 
${\{ \{1, 3\}, \{4\}\}}$ we may write $\{13, 4\}$ or just $13, 4$.  
When we are using hops as indexes, in 
particular, we will simply write $E_{13, 4}$ instead of $E_{\{ \{1,3\} , \{4\} \}}$. 
(In this paper we will not consider examples with $n>9$.)

For any positive integer $n$, we denote the set $\{ 1, \ldots, n \}$ by \hop{n}.  
Given our abuses of notation, we will also write $\hop{n}$ for the hop 
that has $n$ elements of size $1$, namely $\{ \{1\}, \ldots, \{n \} \}$.

A \emph{pure hop} is a hop all of whose elements are of size $1$.  
A \emph{run} is a hop with a single element. (That is, one `runs' so long as there is
no comma, and has to `hop' over any comma.) For instance, $\{ \{1\}, \{3\}, \{4\}\}$ is 
a pure hop, and $\{ \{ 1, 3, 4\} \}$ is a run.  Abusing notation, we will regard a run 
both as a subset of $A$ and as a set with one element (which is a subset of $A$).  
If $A$ is a set with size $n$, then there are $2^n$ pure hops and $2^n-1$ runs on $A$.
Given a hop $H$, there is a natural way to make it into a pure hop, which we denote 
$\ph{H}$, and a natural way to make it into a run, which we denote $\run{H}$.  For 
instance, $\ph{13,4} = \{1,3,4\}$ and $\run{13,4} = \{134 \}$.

Given a set $A$, a hop $H$ across $A$, and a subset $I$ of $A$, the notation $I \in H$ 
has the usual meaning. If $k \in A$, we will use the notation $k \in \in H$ to mean that 
there exists $I \in H$ such that $k \in I$.  
The negation of $k \in \in H$ will be written as $k \notin \in H$.  
Given two elements $i, j \in \in H$, we say that $i$ and $j$ are 
\emph{together}\index{together} in $H$ if there exists $I \in H$ 
such that $i, j \in I$; otherwise we say that $i$ and $j$ are \emph{separate}\index{separate} 
in $H$. For instance, $13 \in \{13, 4\}$,  $3 \notin \{13, 4\}$,  
$3 \in \in \{13, 4\}$, $2 \notin \in \{13, 4\}$;  
$1$ and $3$ are together in $\{13, 4\}$; $1$ and $4$ are separate in $\{13, 4\}$.
A hop $H$ across $A$ is called \emph{complete}\index{complete} if 
$j \in \in H$ for all $j \in A$. Thus a complete hop across $A$ is
a partition of $A$.  

Two hops $H_1$ and $H_2$ are called \emph{disjoint}\index{disjoint} if there is no 
element $i$ such that $i \in \in H_1$ and $i \in \in H_2$.  Given two such hops we 
denote the hop $H_1 \cup H_2$ by $H_1, H_2$ (inserting a comma). 

Given two runs $R_1$ and $R_2$ which are disjoint, we define $R: = R_1R_2$ as the run $R$ whose only element is 
the union of the only elements in $R_1$ and $R_2$.  Compare: if $R_1 = \{\{1\}\} = 1$ and $R_2 = \{\{2,3\} \}= 23$, 
then $R_1, R_2 = \{ \{1\}, \{ 2,3\}\} = 1, 23$ and $R_1R_2 = \{\{1,2,3\}\} = 123$.  

Let $H$ be a hop in $A$ and let $i$ be any element of $A$ such that $i \in \in A$.  We define $H \backslash i$ as 
the hop obtained from $H$ when the subset $I \in H$ such that $i \in I$ has been replaced with 
$I \setminus \{ i \}$, or removed if $I = \{ i \}$.  We define $H \backslash i, j$ 
as $( H \backslash i) \backslash j $.  For instance, if $H = \{13, 4\}$, then $H \backslash 3 = \{1, 4\}$, 
whereas $H \backslash 4 = \{13\}$. 

\subsection*{Definition of $n$-fold vector bundle}
\addcontentsline{toc}{subsection}{\protect\numberline{}Definition of $n$-fold vector bundle}

\begin{df}
\label{df:nvb}
Let $n$ be a non-negative integer.  An \emph{$n$-fold vector bundle}  
consists of a smooth manifold $E_H$ for every pure hop $H$ across $\hop{n}$, together with a 
vector bundle structure on $q^{H,i}_H \co E_{H,i} \to E_H$ for every pure hop $H$ and 
$i \notin \in H$, such that 
$$
\dvbs{E_{H,i,j}}{E_{H,i}}{E_{H,j}}{E_H}
$$ 
is a double vector bundle for every pure hop $H$, and every $i, j \notin \in H$.  

We refer to the whole structure (that is, to the $n$-fold vector bundle) as $E$.
The \emph{total space of} $E$ is $E_{\hop{n}}$, and $M : = E_{\emptyset}$ is the 
\emph{final base}.  
\end{df}

Later in the section we will extend this definition to allow index sets other than $\hop{n}$. 

\begin{rems}  
(i) The commas in the subscripts are important: we are using pure hops (not runs) as indices. 
The various cores associated with the structure will be labeled by hops which are not pure. 

(ii)  The concept of ``$n$-fold vector bundle'' differs from the concept 
of ``$n$-vector bundle'' that appears in category theory. 
We refer to \nvbs{n} generically as \emph{multiple structures} to distinguish them from 
``higher vector bundles''. 

(iii) It might seem that Definition \ref{df:nvb} should include more compatibility conditions. 
For instance, we could require that certain maps in the structure of an \nvb{n} form a 
morphism of \nvbs{k} for $k < n$.  Such conditions are all implied by the current 
definition. 

(iv) In \cite{Gracia-SazM:2009} we added a further condition to the definition, 
namely that a certain combination of the bundle projections is a surjective submersion. 
This turns out to be implied by the rest of 
the definition, as is discussed after Definition \ref{df:decomp} 
below.

(v) A $0$-fold vector bundle is just a manifold. 
A $1$-fold vector bundle is a vector bundle in the usual sense. 
\end{rems}

According to Definition \ref{df:nvb}, a $2$-fold vector bundle is precisely 
a double vector bundle as defined in \cite{Gracia-SazM:2009} and references there.  

Applying Definition \ref{df:nvb} with $n = 3$, a $3$-fold vector bundle consists of 
a commutative diagram as in Figure \ref{fig:triple}(a), such that every two-dimensional 
face forms a double vector bundle.  Again, this definition of \nvb{3} agrees with the 
definition of triple vector bundle \cite{Mackenzie:2005dts,Gracia-SazM:2009}. 

In general, the $2^n$ manifolds that constitute an \nvb{n} can be arranged as the 
vertices of an $n$--dimensional 
cube in a commutative diagram. We refer to this as the \emph{outline} of the \nvb{n}. 

\begin{ej}
Let $E$ be an \nvb{n}. Then $TE$ has a natural structure of an \nvb{(n+1)}. To
specify this, first apply the tangent functor to every map in the outline of $E$; this 
produces an \nvb{n} with final base $TM$. Next, for every pure hop $H$ 
across $\hop{n}$, define $(TE)_{H,n+1} : = T(E_H)$ with its structure as tangent
bundle of $E_H$. The result is an \nvb{(n+1)}, called the \emph{tangent
prolongation of $E$}. 
\end{ej}

\begin{ej}  \label{ej:dec}
Let $M$ be a manifold and let $n$ be a positive integer. Suppose given, for every 
run $R$ across 
$\hop{n}$, a vector bundle $E_R \to M$. From this data we will construct an \nvb{n}. 

For every pure hop $H$ across $\hop{n}$ write  $E_H$ for the pullback manifold
$\PB{R}E_R$ where the pullback is taken over all runs $R$ such that 
$R \subseteq \run{H}$. Suppose that $H_1$ and $H_2$ are two pure hops related by 
$H_1 = H_2, i$ for some $i \in \hop{n}$. Form the Whitney sum vector bundle 
$W := \oplus_S E_S$ on $M$, where the $\oplus$ is over all runs $S$ across $\hop{n}$ 
with $i\in\in S$. 
Now the inverse image vector bundle of $W\to M$ across the projection 
$E_{H_2}\to M$ gives $E_{H_1}$ a vector bundle structure on base $E_{H_2}$. 

In this way we have constructed an \nvb{n} with total space the manifold 
$$
E_{\hop{n}} = \PB{}\{E_R\st R\text{ a run across } \hop{n}\}.
$$
This is the \emph{decomposed $n$-fold vector bundle constructed from the $E_R$}. 
Although $E_{\hop{n}}$ may be considered the Whitney sum of all the $E_R$, this
is not part of the structure of the \nvb{n}, and is usually not relevant. 
\end{ej}

\begin{df}
Let $E$ and $F$ be two \nvbs{n}.  A \emph{morphism of $n$-fold vector bundles} 
$\phi\co E \to F$ consists of a set of smooth maps 
$\phi_H \co E_H \to F_H$ for every pure hop $H$ across $\hop{n}$, 
such that, for every pure hop $H$ and every $i \notin \in H$, the following
\begin{equation}
\label{eq:morph}
\dvb{E_{H,i}}{E_H}{F_{H,i}}{F_H}{}{\phi_{H,i}}{\phi_H}{}
\end{equation}
is a morphism of vector bundles.  An \emph{isomorphism of $n$-fold vector bundles} 
is a morphism which is a
diffeomorphism. 
\end{df}

Note that (\ref{eq:morph}) is a diagram of a morphism, not of a double structure. 

Various alternative formulations of Definition \ref{df:nvb} are possible. 
Consider an \nvb{n} $E$ and let $j \in \hop{n}$. Write $E^{(j)}$ for the 
\nvb{(n-1)} consisting of the manifolds $E_H$ for pure hops $H$ such that 
$j \in \in H$ and the maps between them. In a similar way, write $E_{(j)}$ for the 
\nvb{(n-1)} consisting of the manifolds 
$E_H$ for pure hops $H$ such that $j \notin \in H$ and the maps between them. 
The rest of the structure on the \nvb{n} $E$ can be described as a morphism of 
\nvbs{(n-1)} from $E^{(j)}$ to $E_{(j)}$. 

By reversing this description, an $n$-fold vector bundle can be recursively 
defined as a ``vector bundle object in the category of $(n-1)$-fold vector bundles''. 


\subsection*{The cores of an $n$-fold vector bundle}
\addcontentsline{toc}{subsection}{\protect\numberline{}{The cores 
of an $n$-fold vector bundle}}

The core structure of a triple vector bundle was recalled in the Introduction;
see Figure \ref{fig:triple}(b). The approach used there may be extended
to \nvbs{n} for any $n$, and we outline it very briefly. Consider an \nvb{n} $E$, 
and any $1\leq i\neq j\leq n$. Write $H = \hop{n}\backslash i,j$. Then 
$E_{\hop{n}} = E_{i, j, H}$ is a
double vector bundle with side bundles $E_{H,i}$ and $E_{H,j}$ and final base
$E_H$. As such it has a core, which we denote $E_{ij,H}$, and which is a vector 
bundle on base $E_H$. Further, for each $k\in\hop{n}$, $k\neq i,j$, the vector
bundle structure on $E_{\hop{n}}\to E_{\hop{n}\backslash k}$ restricts to give
$E_{ij,H}$ a vector bundle structure on base $E_{ij,H\backslash k}$. Thus $E_{ij,H}$
is an \nvb{(n-1)}. 

This process continues inductively until we reach the ultracore, which we need
throughout the rest of the paper. 

\begin{df}
Let $E$ be an \nvb{n}. The \emph{ultracore}\index{ultracore} of $E$ is the 
set of elements $e \in E$ such that, for every $i, j  \in \underline{n}$,
with $i \neq j$,  the element $q^{\underline{n}}_{\underline{n} \setminus i} (e)$ 
is a zero of the vector bundle $E_{\underline{n} \setminus i}  \to
E_{\underline{n} \setminus i, j}$.
\end{df}

We denote the ultracore of $E$ by $E_{\run{\hop{n}}}$ 
or by $C(E)$. As in the triple case, the $n$ vector bundle structures on 
$E_{\run{\hop{n}}}$ coincide and make the ultracore a vector bundle on $M$. 

We will also need the ultracores of various substructures of $E$. Let $R$ be
a run across $\hop{n}$ and write $H = \ph{R}$. Assume that $H$ is not $\hop{n}$
(we have already considered that case), and that it contains more than one element. 
If $H$ has $k$ elements, $1 < k < n$, then $E_H$ is a \nvb{k} with respect to the
structure induced from $E$. Write $E_R$ for the ultracore of $E_H$. 

We now have, for each run $R$ across $\hop{n}$, a vector bundle $E_R$ on base $M$,
which is the ultracore of $E_{\ph{R}}$. To avoid trouble with extreme cases, we define 
the ultracore of a 1-fold vector bundle to be the vector bundle itself, 

\begin{rem}
There is a complicated system of cores of substructures lying between the
$E_R$ just defined, and the cores of the double vector bundles which have $E_{\hop{n}}$
as total space. We describe some of these in the next subsection. 
\end{rem}

\begin{df}  \label{defin:bb}
Let $E$ be a \nvb{n}.  The \emph{building bundles} of $E$ are the vector bundles 
$E_R \to M$ for all runs $R$ across $\hop{n}$. The set of building bundles is
denoted $\BE$. 
\end{df}

In particular the side bundles $E_i$, $1\leq i\leq n$, with base $M$, are 
building bundles. Altogether there are $2^n-1$ building bundles.  

Applying the construction of Example \ref{ej:dec} to the building bundles yields an
\nvb{n} which we call the \emph{decomposed form}\index{decomposed form} of $E$ and denote 
by $\decom{E}$. 

\begin{df}  
\label{defin:stato}
Let $E$ and $F$ be two \nvbs{n} which have the same building bundles.  
A \emph{statomorphism}\index{statomorphism} $\phi\co  E \to F$ is a morphism of 
\nvb{n} which induces the identity on all building bundles.  
\end{df}

A statomorphism is necessarily an isomorphism; indeed any morphism of \nvbs{n}  
which induces an isomorphism of vector bundles on all building bundles is an 
isomorphism of \nvbs{n}.  

Statomorphisms and building bundles are essential for the calculation 
in Section \ref{sec:thetaX} of the duality functor groups. 

\subsection*{The duals of an $n$-fold vector bundle}
\addcontentsline{toc}{subsection}{\protect\numberline{}{The duals of an $n$-fold vector bundle}}

Let $E$ be an \nvb{n}. The total space $E_{\hop{n}}$ has $n$ distinct structures
of vector bundle. Take $i\in\hop{n}$ and consider the dualization of
$E_{\hop{n}}$ as a vector bundle over $E_{\hop{n}\setminus i}$. 
We must show that dualization does lead to another \nvb{n}. 
There are two aspects to the problem. 

First, for each $j\neq i$, $j\in\hop{n}$, there is a double vector bundle for 
which $E_{\hop{n}}$ is the total space and $E_{\hop{n}\setminus i}$ is a side bundle. 
See Figure \ref{fig:topij}(a). We denote the core by $E_{H,ij}$ where 
$H = \hop{n}\setminus i,j$. Recall that 
for the core of a double vector bundle which is contained within a multiple structure, our
rule for notation is to combine the two indices which distinguish the side bundles, leaving 
unchanged the hop which indexes the final base of the double vector bundle. 

\begin{figure}[h]
\centering
\subfloat[]
{\xymatrix@=15mm{
E_{\hop{n}} \ar[d]\ar[r] & E_{\hop{n}\setminus j}\ar[d]\\
E_{\hop{n}\setminus i} \ar[r] & E_{\hop{n}\setminus i,j}\\
}}
\qquad
\subfloat[]
{\xymatrix@=15mm{
E_{\hop{n}}\duer E_{H,j} \ar[d]\ar[r] & E_{H,ij}\duer E_{H}\ar[d]\\
E_{H,j} \ar[r] & E_{H}\\
}}
\caption{\ \label{fig:topij}}  
\end{figure}

The dual of Figure \ref{fig:topij}(a) over $E_{\hop{n}\setminus i}$ 
is shown in Figure \ref{fig:topij}(b). 
A specific example, with $n = 4$, $i = 1$, $j = 2$, is shown in Figure
\ref{fig:4-1}. 

\begin{figure}[h]
\centering
\subfloat[]{
\xymatrix@=15mm{  
E_{1,2,3,4} \ar[r] \ar[d] & E_{1,3,4} \ar[d] \\
E_{2,3,4} \ar[r]          & E_{3,4} \\
}}
\qquad
\subfloat[]{
\xymatrix@=15mm{  
E_{1,2,3,4}\duer E_{2,3,4} \ar[r] \ar[d] & E_{12,3,4}\duer E_{3,4} \ar[d] \\
E_{2,3,4} \ar[r]          & E_{3,4} \\
}}
\caption{\label{fig:4-1}}
\end{figure}

In general there will be $(n-2)$ further double vector bundles to consider. 
In the case $n = 4$ these are shown with their duals over $E_{2,3,4}$ in
Figures \ref{fig:4-2} and \ref{fig:4-3}.  

\begin{figure}[h]
\centering
\subfloat[]{
\xymatrix@=15mm{  
E_{1,2,3,4} \ar[r] \ar[d] & E_{1,2,4} \ar[d] \\
E_{2,3,4} \ar[r]          & E_{2,4} \\
}}
\qquad
\subfloat[]{
\xymatrix@=15mm{  
E_{1,2,3,4}\duer E_{2,3,4} \ar[r] \ar[d] & E_{13,2,4}\duer E_{2,4} \ar[d] \\
E_{2,3,4} \ar[r]          & E_{2,4} \\
}}
\caption{\label{fig:4-2}}
\end{figure}

\begin{figure}[h]
\centering
\subfloat[]{
\xymatrix@=15mm{  
E_{1,2,3,4} \ar[r] \ar[d] & E_{1,2,3} \ar[d] \\ 
E_{2,3,4} \ar[r]          & E_{2,3} \\
}}
\qquad
\subfloat[]{
\xymatrix@=15mm{  
E_{1,2,3,4}\duer E_{2,3,4} \ar[r] \ar[d] & E_{14,2,3}\duer E_{2,3} \ar[d] \\
E_{2,3,4} \ar[r]          & E_{2,3} \\
}}
\caption{\label{fig:4-3}}
\end{figure}

We now have four vector bundle structures on $E_{1,2,3,4}\duer E_{2,3,4}$, with bases 
respectively $E_{2,3,4}$, $E_{12,3,4}\duer E_{3,4}$, $E_{13,2,4}\duer E_{2,4}$
and $E_{14,2,3}\duer E_{2,3}$. The triple vector bundle structure of the first
remains unchanged in this dualization. We need to show that the other
three have natural triple vector bundle structures and that these
`fit together' correctly. 

First consider $E_{12,3,4}\duer E_{3,4}$. The bundle of which this is the dual
is a triple vector bundle in a natural way. Namely, $E_{12,3,4}$ is the core of 
the double vector bundle in Figure \ref{fig:4-1}(a). The two other vector
bundle structures on $E_{1,2,3,4}$, on bases $E_{1,2,3}$ and $E_{1,2,4}$,
restrict to give vector bundle structures on $E_{12,3,4}$ over bases
$E_{12,3}$ and $E_{12,4}$, and these form the triple vector bundle shown
in Figure \ref{fig:4-tvb}(a). The cores of the upper faces of this triple vector
bundle are (left) $E_{124,3}$, (rear) $E_{123,4}$ and (top) $E_{12,34}$. 
The ultracore is $E_{1234}$. 

\begin{figure}[h]
\centering
\subfloat[]{
{\xymatrix@=3mm{
E_{12,3,4}\ar[rr] \ar[dd] \ar[rd]    & & E_{12,4}\ar'[d][dd] \ar[dr] &\\
& E_{12,3}   \ar[rr] \ar[dd]    & & E_{12} \ar[dd]  \\
E_{3,4}\ar'[r][rr] \ar[dr]  &       & E_4 \ar[dr] &\\
& E_3  \ar[rr] & & M\\
}}}
\qquad
\subfloat[]{
{\xymatrix@=3mm{
E_{12,3,4}\duer E_{3,4}\ar[rr] \ar[dd] \ar[rd]    & 
                                  & E_{123,4}\duer E_4\ar'[d][dd] \ar[dr] &\\
& E_{124,3}\duer E_3   \ar[rr] \ar[dd]    & & E_{1234}^* \ar[dd]  \\
E_{3,4}\ar'[r][rr] \ar[dr]  &       & E_4 \ar[dr] &\\
& E_3  \ar[rr] & & M\\
}}}
\caption{\label{fig:4-tvb}}
\end{figure}

The dual over $E_{3,4}$ is shown in Figure \ref{fig:4-tvb}(b). In the same way
we obtain the two triple vector bundles shown in Figure \ref{fig:4-tvb-2}. 

\begin{figure}[h]
\subfloat[]{
{\xymatrix@=3mm{
E_{13,2,4}\duer E_{2,4} \ar[rr] \ar[dd] \ar[rd]    & & 
               E_{123,4}\duer E_4 \ar'[d][dd] \ar[dr] &\\
& E_{134,2}\duer E_2  \ar[rr] \ar[dd]    & & E_{1234}^* \ar[dd]  \\
E_{2,4}\ar'[r][rr] \ar[dr]  &       & E_4 \ar[dr] &\\
& E_2  \ar[rr] & & M\\
}}}
\
\subfloat[]{
{\xymatrix@=3mm{
E_{14,2,3}\duer E_{2,3}\ar[rr] \ar[dd] \ar[rd]    & 
                                  & E_{134,2}\duer E_2\ar'[d][dd] \ar[dr] &\\
& E_{124,3}\duer E_3   \ar[rr] \ar[dd]    & & E_{1234}^* \ar[dd]  \\
E_{2,3}\ar'[r][rr] \ar[dr]  &       & E_2 \ar[dr] &\\
& E_3  \ar[rr] & & M\\
}}}
\caption{\label{fig:4-tvb-2}}
\end{figure}

We can now complete the outline of the \nvb{4} towards which we are working.
This is shown in Figure \ref{fig:shard}. 

\begin{figure}[h]
\xymatrix@!{  
& E_{1,2,3,4}\duer E_{2,3,4} \ar[r] \ar[d] \ar[ddl] \ar[ddr]  
            & E_{12,3,4}\duer E_{3,4} \ar[d] \ar[ddl] \ar[ddr] &\\
& E_{2,3,4} \ar[r] \ar[ddl] \ar[ddr] & E_{3,4} \ar[ddl] \ar[ddr] &\\
E_{14,2,3}\duer E_{2,3} \ar[r] \ar[d] \ar[ddr] & E_{124,3}\duer E_3 \ar[d] \ar[ddr] 
& E_{13,2,4}\duer E_{2,4} \ar[r]\ar[d] \ar[ddl] & E_{123,4}\duer E_4 \ar[d] \ar[ddl]\\
E_{2,3} \ar[r] \ar[ddr]& E_3 \ar[ddr]
& E_{2,4} \ar[r] \ar[ddl] & E_4 \ar[ddl] \\
& E_{134,2}\duer E_2 \ar[r] \ar[d] & E_0 \ar[d] &\\
& E_2 \ar[r] & M &\\
}
\caption{\label{fig:shard}}
\end{figure}

\clearpage

The two double vector bundles in the centre of Figure \ref{fig:shard} have 
been found from their total spaces and the arrangement of the
triple vector bundle $E_{2,3,4}$ which is left unchanged by the dualization. 
Their structure as triple vector bundles can be seen to be consistent with
those in Figure \ref{fig:4-tvb-2}, which were obtained as duals of triple 
vector bundle structures on the cores of double vector bundles with total 
space $E_{1,2,3,4}$. 

There is one `new' triple vector bundle in Figure \ref{fig:shard}, namely
that obtained by deleting the triple vector bundle $E_{2,3,4}$. The proof
that this is a triple vector bundle follows the same pattern as for the
corresponding result in the dual case \cite[5.4]{Mackenzie:2005dts}. 

In the case of general $n$ we would at this stage have reached a
collection of \nvbs{(n-2)} and can proceed by induction. At each stage
we are taking the duals of double vector bundles, as described in
\cite{Gracia-SazM:2009} and elsewhere. The only task is to ensure
consistency. 

We evidently have an \nvb{n} but it is not presented as Definition \ref{df:nvb} 
requires. For this we need to allow index sets other than $\hop{n}$. 
Write $[n] := \{0,1, \ldots, n\}$ and for $i\in\hop{n}$ 
write $[n,i]$ for the set $[n] \setminus \{i\}$.  

The bundles appearing in the structure of $E_{\hop{n}}\duer E_{\hop{n}\setminus i}$ 
are of the form $E_{R, H}\duer E_H$ where $R$ is a run
across $\hop{n}$ and $H$ is a hop across $\hop{n}$. Let $R'$ denote the
run across $[n]$ such that $R, H, R'$ is a complete hop across $[n]$
and write
\begin{equation}
\label{eq:ERH}
E_{R', H} := E_{R, H}\duer E_H  
\end{equation}
(we write $E_{R', H}$ or $E_{H, R'}$ as convenient). For example, for the 
space $E_{13,2,4}\duer E_{2,4}$ in Figure~\ref{fig:shard}, we have $R' = 0$ and so 
$E_{13,2,4}\duer E_{2,4} = E_{0,2,4}.$

With this renaming, $E_{\hop{n}}\duer E_{\hop{n}\setminus i}$ satisfies Definition 
\ref{df:nvb} with the index set now $[n,i]$. In particular, applying this notation 
to Figure \ref{fig:shard} 
shows that $E_{1,2,3,4}\duer E_{2,3,4}$ is a \nvb{4} with index set $\{0, 2, 3, 4\}$. 

Equation (\ref{eq:ERH}) effectively defines $E_{H'}$ for every hop $H'$ across $[n]$. 
Namely, write $H' = R', H$ where $R'$ is the element of $H'$ with $0\in\in R'$ (if
there is no such element, then $E_{H'}$ is already defined) and $H$ is a hop
across $\hop{n}$, possibly empty. Then read (\ref{eq:ERH}) as defining 
$E_{R, H}$. 
Finally, set $E_{[n]} : = E_{\emptyset} = M$.

Let $S_{n+1}$ be the group of permutations of $[n]$. Write $\lambda_i$ for the
transposition $(0\,i)$ in $S_{n+1}$.  Then $\lambda_i$ defines a bijection 
between $[n,0] = \hop{n}$ and $[n,i]$, and therefore induces a bijection between 
the hops across $\hop{n}$ and the hops across $[n,i]$; we denote this bijection
also by $\lambda_i$.

Define an \nvb{n} $E^{X_i}$ by the formula 
\begin{equation}
(E^{X_i})_H : = E_{\lambda_i(H)}
\end{equation}
for every pure hop $H$ across $\hop{n}$. 
To summarize:

\begin{thm}  \label{thm:dualdef}
$E^{X_i}$ is a well-defined \nvb{n}. 

The cores of $E^{X_i}$ are given by the same formula 
$(E^{X_i})_H = E_{\lambda_{i} (H)}$ 
where now $H$ is any hop across $\hop{n}$.
\end{thm}

We call $E^{X_i}$ the \emph{$X_i$-dual of $E$} or the \emph{$i$-dual of $E$}.  

Notice that if we only want to describe the manifolds that appear in the definition 
of $E^{X_i}$, then they are entirely determined by the permutation 
$\lambda_i = (0\,i) \in S_{n+1}$.  In particular $\lambda_i$ permutes the vector 
bundles $E_1, \ldots, E_n$ and the dual $E_0$ of the ultracore.

\section{The duality functor groups}
\label{sect:dfg}

We have defined the action of dualization on multiple vector bundles. The reader
will appreciate however, that composing several dualizations using this formulation
leads to unwieldy expressions which are difficult to recognize. 
Further, as \cite{Gracia-SazM:2009} showed, it is actually not possible to tell from 
diagrams of 
the outlines whether two multiple vector bundles are canonically isomorphic (as defined 
below). We therefore need to extend the techniques of \cite{Gracia-SazM:2009} to the
$n$-fold case. The key idea is to consider the effect of dualization not only on the
multiple vector bundles but on the maps between them. This is a simple idea, but it
is worth drawing attention to, since the role it plays does not arise for duality
of ordinary vector bundles. 

To be able to dualize, in any direction, a map between two $n$-fold vector bundles,
it is necessary that the dualizations take place over maps which are isomorphisms. 
We ensure this by considering only statomorphisms. 

\subsection*{Statomorphism categories}
\addcontentsline{toc}{subsection}{\protect\numberline{}{Statomorphism categories}}

Let $\C_n$ be the category whose objects are $n$-fold vector bundles 
and whose morphisms are statomorphisms of $n$-fold vector bundles.    
Let the opposite category be denoted by $\C_n^{op}$ or $\C_n^{-1}$.  
The dualization operators $X_k$ defined above extend in a natural way to
functors $\C_n \to \C_{n}^{op}$, also denoted $X_k$. Since every morphism 
in $\C_n$ is invertible, we can also consider them as functors 
$X_k\co \C_n^{op} \to \C_n$.  It is therefore possible to compose such 
functors. Write $\wg_n$ for the group generated by the $X_1,\dots,X_n$. 

As noted after Theorem \ref{thm:dualdef}, every dualization functor produces a 
permutation of the vector bundles $E_1, \ldots, E_n, E_0$.  
Thus there is a surjective group homomorphism $\pi\co  \wg_n \to S_{n+1}$.

\subsection*{Decomposed $n$-fold vector bundles}
\addcontentsline{toc}{subsection}{\protect\numberline{}{Decomposed $n$-fold vector bundles}}

Recall that, for an \nvb{n} $E$, the decomposed form of $E$, constructed as in 
Definition \ref{defin:bb}, is denoted $\decom{E}$. 

\begin{df}
\label{df:decomp}
A \emph{decomposition} of an \nvb{n} $E$ is a statomorphism onto its decomposed 
form $E \to \decom{E}$.
\end{df}

Grabowski and Rotkiewicz \cite{GrabowskiR:2009} proved that every \nvb{n} has a
decomposition. Rather than consider duality operators on an arbitrary \nvb{n}, 
it is therefore sufficient to consider the decomposed case. 

This situation is analogous to that of trivializations of an ordinary vector bundle.
If a vector bundle is, say, flat and on a simply-connected base, then a trivialization
exists, but is not unique (assuming that the rank is positive). Rather than consider
the group of automorphisms of the vector bundle itself, one considers the group of
automorphisms of a trivial vector bundle of the same rank. For multiple vector bundles
a decomposition always exists, but is not unique, and we consider automorphisms not
of the given multiple vector bundle, but of its decomposed form. 


Write elements $e\in \decom{E}$ in the form $e = e_I$ where $I$ ranges over all 
non-empty subsets of $\hop{n}$. 

A statomorphism $\decom{E}\to\decom{E}$ consists of a set of multilinear maps, which 
we now describe. For ease of
notation, we write their domains as tensor products, but their arguments as strings of vectors.

Throughout the rest of the paper, 
`statomorphism' always refers to a statomorphism which is an automorphism of a 
decomposed \nvb{n}. Recall the case $n = 3$ from \cite{Gracia-SazM:2009}. 

\begin{ej}
\label{ej:stat3}
Let $E$ be a triple vector bundle. The elements of $\decom{E}$ are 
strings \newline
$(e_1, e_2, e_3, e_{12}, e_{23}, e_{31}, e_{123})$
with $e_I\in E_I$ for each nonempty $I\subseteq\{1,2,3\}$. 

A statomorphism $\phi\co \decom{E}\to\decom{E}$ is of the form 
\begin{multline*}
\phi(e_1, e_2, e_3, e_{12}, e_{23}, e_{31}, e_{123}) = 
(e_1, e_2, e_3, \\
e_{12} +\phi_{12}(e_1, e_2), e_{23}+\phi_{23}(e_2, e_3), e_{31}+\phi_{31}(e_3, e_1), \\
e_{123} + \phi_{12,3}(e_{12}, e_3) + \phi_{23,1}(e_{23}, e_1) + \phi_{31,2}(e_{31}, e_2) + \phi_{123}(e_1, e_2, e_3)) 
\end{multline*}
where $\phi_{ij}\co E_i\otimes E_j\to E_{ij}$, $\phi_{ij,k}\co E_{ij}\otimes E_k\to E_{ijk}$, and
$\phi_{123}\co E_1\otimes E_2\otimes E_3\to E_{123}$ are multilinear maps. 

These $\phi_I$ can also be written as elements of 
$E_1\otimes E_2\otimes E_{12}^*$, $E_2\otimes E_3\otimes E_{23}^*,$ \dots, 
$E_1\otimes E_2\otimes E_{3}\otimes E_{123}^*$, 

From (\ref{eq:ERH}) it follows that for any $I \subseteq \hop{n}$, we 
have $E_{I}^* : = E_{I^C}$, where $I^C : = [n] \setminus I$.

The seven tensor products can therefore be written as 
$E_1\otimes E_2\otimes E_{03}$, $E_2\otimes E_3\otimes E_{01},$ \dots, 
$E_1\otimes E_2\otimes E_{3}\otimes E_0$, These correspond to the seven partitions
of the set $\{0,1,2,3\}$ into three or more subsets. 
\end{ej}

In general, a statomorphism adds a term to $e_I$ for each partition of $I$ 
into two or more nonempty sets. We now formalize this statement. 

A partition of $[n]$ is the same as a complete run across $[n]$, and 
we will continue to use the notation set up in \S\ref{sect:not} for runs, 
For every positive integer $n$, denote by $\parti{n}$ the set of all 
partitions of $[n]$ into three or more subsets.  
The number of ways in which $[n]$ can be partitioned into $k\geq 2$ subsets is
the Stirling number of the second kind \cite{Knuth:1}, 
$$
\stsec{n+1}{k} = \frac{1}{k!}\sum_{j=1}^k (-1)^{k-j}\binom{k}{j}j^{n+1}.
$$
Thus the number of components in a statomorphism of $\Bar{E}$ is 
$
\sum_{k=3}^{n+1} \stsec{n+1}{k}. 
$ 
The values for $n = 3,4,5,6$ are $7$, $36$, $171$ and $813$. 

For any $P \in \parti{n}$, define 
\begin{equation}
\label{eq:VP}
\V{P} : = \Gamma \left( \bigotimes_{I \in P} E_I^{*} \right)
\end{equation}

Elements of a given $\V{P}$ can be interpreted in various ways. 
For example, consider $\phi \in \V{012,3,4}$. This $\phi$ is a section 
of $E_{012}^* \otimes E_{3}^* \otimes E_{4}^*$, or
a map from $E_{012} \otimes E_3 \otimes E_4$ to the real line bundle. 
However, since $E_{012} = E_{34}^*$, we can also regard $\phi$ as a map from 
$E_3 \otimes E_4$ to $E_{34}$.  Similarly, we can also think of $\phi$ as a map from 
$E_{3}$ to $E_{34} \otimes E_{0123}$, and so on. We will use the same letter to refer 
to all these maps, according to which is most convenient at a particular time.   
For instance, given elements $e_I \in E_I$ for each run $I$, if
we write $\phi(e_3, e_4)$, it should be understood that we are thinking of 
$\phi$ as $E_3 \otimes E_4 \to E_{34}$ and that $\phi(e_3, e_4) \in E_{34}$.

\begin{df}
\label{df:ps}
The \emph{parameter space}\index{parameter space} 
$\G{\BE}$ of the set of building bundles $\BE$ is defined as
$$
\G{\BE} : = \bigoplus_{P \in \parti{n}} \V{P}
$$
\end{df}

We will see shortly that $\G{\BE}$ can be canonically identified with the group of 
statomorphisms $\Dec \decom{E}$ and, moreover, with the group of statomorphisms 
$\Dec \decom{E}^W$ for any dualization functor $W$. 

Take a word $W$ in $X_1,\dots, X_n$; that is, $W\in\wg_n$ and write 
$k : = \pi(W) (0)$, where $\pi$ is the surjection $\wg_n\to S_{n+1}$. 
Then 
\begin{equation*}
\decom{E}^W = \decom{E^W} = \bigoplus_{\substack{R \textrm{ a run} \\ \textrm{across } [n,k] }}  E_R
\end{equation*}
Write elements $e \in \decom{E}^W$ as tuples $(e_R)$ where $e_R \in E_R$ for each run $R$.  
Now define a map
\begin{equation}  
\label{eq:defgk}
\Omega^k \co \G{\BE} \longrightarrow  \Dec \decom{E}^W. 
\end{equation}  
Take $\phi \in \G{\BE}$ and write $\phi = (\phi_P)$ where $P$ runs through $\parti{n}$. 
Take $e \in \decom{E}^W$ and define $f : = \Omega^k(\phi) (e)$ by the equations:
\begin{equation}  \label{eq:stato}
f_R \; = \; e_R + \; \sum \phi_P(e_{J_1}, \ldots, e_{J_m})
\end{equation}
for each run $R$ across $[n,k]$. The sum is over all $P \in \parti{n}$ such 
that $R^c \in P$ and the $J_i$ are the remaining elements of $P$; that is, 
$P = \{R^c, J_1, \ldots, J_m\}$. The proof of the following is now a matter of unwinding the
statement. 

\begin{thm}  \label{thm:statogroup}
Let $W \in \wg_n$.  Then the map $\Omega^k$ defined in \eqref{eq:stato} is a well-defined, 
canonical bijection between the set $\G{\BE}$ and the group $\Dec \decom{E}^W$. 
\end{thm}

For the case $n=3$, see Section 4 of \cite{Gracia-SazM:2009}. The treatment here
corrects some problems with the account in \cite{Gracia-SazM:2009}. In what follows we 
will not need the group structures on the various $\Dec \decom{E}^W$. 

\subsection*{The basic theorem} 
\addcontentsline{toc}{subsection}{\protect\numberline{}{The main theorems}}

We now give the result which is the foundation of the subsequent calculations, Theorem
\ref{thm:W=1}: an element of
$\wg_n$ which induces the identity element of $S_{n+1}$ is naturally isomorphic to the 
identity under statomorphisms if and only if it acts trivially on $\G{\BE}$. 
This provides a concrete and explicit calculation that determines when two dualization 
functors are naturally isomorphic.  

To begin, for each run $R \subseteq [n]$ let $E_R \to M$ be a vector bundle, subject to the conditions $E_{\emptyset} = M$ and $E_{R^C} = E_R^*$.  
We fix this data for the rest of the section.  

Denote by $\BE$ the set of bundles $\{ E_I \; \vert \; I \subseteq  \hop{n} \}$ and pick an \nvb{n} $E$ which has $\BE$ as its building bundles.   
As usual we denote by $\decom{E}$ the decomposed form of $E$.  (Of course $\decom{E}$ can be constructed from  $\BE$ and vice versa.)   
Take $W \in \wg_n$ and write $k : = \pi(W)(0)$.  The decomposed form of $E^W$ is $\decom {E^W} = \decom{E}^W$.  
We will denote the building bundles of $E^W$ by $\BE^W$.  Notice that, as a set, $\BE^W =\{  E_I \; \vert \; I \subseteq [n,k] \}$.  
Finally, let $\varepsilon_W$ be $+1$ or $-1$  depending on the parity of $\pi(W)$.

We now define the crucial action.   Let $\phi \in \G{\BE}$ and pick a decomposition $S_1$ of $E$.  
Then $S_2 : = S_1 \circ \Omega^0(\phi)$ is another decomposition of $E$, with $\Omega^0$ as defined in \eqref{eq:stato}.    
Notice that $(S_1^W)^{\varepsilon_W}$ and $(S_2^W)^{\varepsilon_W}$ are decompositions of $E^W$.  
Hence there is a unique $\psi \in \G{\BE}$ such that $S_2^{\varepsilon_W} = \Omega^k(\psi) \circ S_1^{\varepsilon_W}$.   
We now have the situation shown in Figure~\ref{fig:commdiag}. 

\begin{figure}[h] 
$$
\xymatrix{
E \ar[r]^{S_1} \ar[rd]_{S_2} & \decom{E} \ar[d]^{\Omega^0(\phi)}  & & \decom{E^W}  \ar[d]_{\Omega^k(\psi)} 
  & E^W \ar[l]_{(S_1^W)^{\varepsilon_W}}   \ar[dl]^{(S_2^W)^{\varepsilon_W}}
  \\
& \decom{E} 
& & \decom{E^W  }
}
$$
\caption{\label{fig:commdiag}}
\end{figure}

Then we define the map
$\theta^{\BE}_W : \G{\BE} \to \G{\BE}$ by  $\theta^{\BE}_W (\phi) : = \psi$. 
We have the following results:
\begin{thm} \
\begin{enumerate}
\item  The map $\theta^{\BE}_W$ is well defined.  Specifically, $\theta^{\BE}_W(\phi)$ depends only on the building bundles $\BE$, on $W$ and 
on $\phi \in \G{\BE}$; it does not depend on the choice of \nvb{n} $E$ or on the choice of decomposition $S_1$.
\item  If $W_1, W_2 \in \wg_n$, then  
\begin{equation} \label{eq:quasiaction}
\theta^{(\BE^{W_1})}_{W_2} \circ \theta^{\BE}_{W_1} = \theta^{\BE}_{W_2 W_1}.
\end{equation}
\end{enumerate}
\end{thm}
The proof follows the same pattern as for the triple case \cite{Gracia-SazM:2009}. 
The first part is proved by a direct computation in Section \ref{sec:thetaX}.  
The second part follows from the definition, 
keeping in mind that $\Dec \decom E$ and $\Dec \decom E^W$ are groups.  

Equation \eqref{eq:quasiaction} 
can be interpreted as a groupoid action, but we will not do so here.

Next, notice that the set of building bundles $\BE^W$ depends only on $\BE$ and $\pi(W)$.    
In particular, if $\pi(W)$ is the identity, then $\BE^W = \BE$.  Define
\begin{equation}
\label{eq:ktilde}
\Tilde{K}_{n+1}:= \{W\in\wg_n\;\vert \; \pi(X) = \id\}.  
\end{equation}
We can therefore define an action 
$$
\theta\co  \Tilde{K}_{n+1} \times \G{\BE} \longrightarrow \G{\BE}\qquad\text{by}\qquad
\theta_W (\phi) : = \theta_W^{\BE} (\phi).
$$ 
This action, as anticipated, is faithful, as given by the next theorem.

\begin{thm}
\label{thm:W=1}
Let $W_1, W_2 \in \wg_n$ such that $\pi(W_1) = \pi(W_2)$.  Then the functors $W_1$ and $W_2$ 
are naturally isomorphic to each other if and only if 
$\theta^{\BE}_{W_1} = \theta^{\BE}_{W_2}$.  
\end{thm}

The proof of this theorem is identical to the case of double and triple vector bundles 
(see \cite[2.7]{Gracia-SazM:2009}), so we omit it. We can finally define the
group which is the subject of the paper. 

\begin{df}
\label{df:dg}
The group $\dg_n$, called the \emph{dualization functor group for \nvbs{n},} is
the quotient of $\wg_n$ over natural isomorphism in $\C_n$. 
\end{df}

From now on, we write $X_i$ for the element of $\dg_n$ corresponding to the 
dualization operator $X_i$, and so on. 

The morphism $\pi$ descends to a surjective morphism $\dg_n\to S_{n+1}$, also
denoted $\pi$. The kernel is the quotient of $\Tilde{K}_{n+1}$ over natural
isomorphism, and we denote it $K_{n+1}$. 

\begin{cor}
\label{cor:faithfulK}
The restriction of $\theta$ to the group $K_{n+1}$  is faithful on the set $\G{\BE}$. 
\end{cor}

The corollary allows us to identify $K_{n+1}$ with its action on $\G{\BE}$, that is, with 
a subgroup of the symmetric group on $\G{\BE}$. 

We need one final lemma preparatory to our calculation of $K_{n+1}$.  There is an evident
short exact sequence, 
\begin{equation}  
\label{eq:ses}
\xymatrix{1 \ar[r] & K_{n+1} \ar[r] & \dg_n \ar[r]^{\pi}  & S_{n+1} \ar[r] & 1.}
\end{equation}
and we know that $\dg_n$ is generated by $X_1, \ldots, X_n$.  We also know that 
$\pi (X_k) = (0\,k)$ and that the $(0\,k)$ generate $S_{n+1}$. The following lemma is an easy 
exercise in algebra. 

\begin{lem} \label{lem:gener}
Let $\pi\co G \to S$ be a surjective group homomorphism.  Let $g_1,\ldots,g_n$ be a set of 
generators of $G$ and write $\sigma_i:=f(g_i)$.  Let 
$\{ R_j(\sigma_1,\ldots,\sigma_n) \; \vert \; j=1,\ldots,m \}$ be a set of relations for 
a presentation of $S$ with generators $\sigma_1,\ldots,\sigma_n$. 
Then the kernel of $\pi$ is the normal subgroup of $G$ generated by 
$\{ R_j(g_1,\ldots,g_n) \; \vert \; j=1,\ldots,m  \}$.
\end{lem}

We use the standard presentation of $S_{n+1}$ with generators $\sigma_k : = (0\,k)$
by relations $\sigma_i^2$, $(\sigma_i\sigma_j)^3$, and 
$(\sigma_i\sigma_j\sigma_i\sigma_k)^2$ with $i,j,k \in\hop{n}$ distinct. 
Hence, $K_{n+1}$ is the normal subgroup of $\dg_n$ generated by 
\begin{equation}  \label{eq:kernelwords}
 X_i^2,\, (X_iX_j)^3,\, (X_iX_jX_iX_k)^2 
\end{equation}
for all distinct values $i, j, k\in\hop{n}$. From ordinary duality
and the case $n = 2$ we know that $X_i^2 = (X_iX_j)^3 = 1$. 
In \S\ref{sect:tgk} we will find the order of $K_{n+1}$. 


\section{Calculation of the action $\theta$ of $K_{n+1}$ on $\G{\BE}$}  
\label{sec:thetaX}

In this section we calculate the action of $K_{n+1}$ on $\G{\BE}$. The first step is to
compute 
$\theta_k := \theta_{X_k}^{\BE}$ for $k \in\hop{n}$. We begin by recalling the 
action in the triple case. 

\begin{ej}
\label{ej:X3}
The statomorphism $\phi$ in Example \ref{ej:stat3} consists of seven maps; in
accordance with the notation of (\ref{eq:VP}), denote these by
\begin{equation}
\label{eq:list3}
(1,2,03), \ (2,3,01),\ (3,1,02),\ (12,3,0),\ (23,1,0),\ (31,2,0),\ (1,2,3,0). 
\end{equation}
Here we are suppressing $\phi$, rather as if we were to denote the entries of 
a matrix $(a_{ij})$ just by $ij$. In \cite{Gracia-SazM:2009} these were denoted,
respectively, by $\gamma, \alpha, \beta, \nu, \lambda, \mu, \rho.$

We found $\theta_X$ of these to be, respectively, 
\begin{gather*}
-(2,13,0),\ (2,3,01),\ -(3,12,0),\ -(3,02,1),\ -(23,1,0),\ (31,2,0),\\
 -(0,2,3,1) + (1,2,03)\,(0,3,12) + (1,3,02)\,(0,2,13). 
\end{gather*}
\end{ej}

For clarity, we introduce a neologism and call an element of a $V_P$ 
for some $P\in\parti{n}$, a \emph{tomo}. A tomo is a statomorphism in its
own right and so $\theta_X$ acts upon it. Further, for
$\mu\in V_P$ 
where $P$ is a partition into $m$ subsets, call $\mu$ an \emph{$(m-1)$-tomo.}
Thus $(m-1)$ is the number of commas in the explicit expression of $\mu$. 

Example \ref{ej:X3} illustrates two general features: for a 2-tomo $\mu$, 
any $\theta_X(\mu)$ is $\pm\mu'$ for another 2-tomo
$\mu'$. Secondly, for a $m$-tomo $\mu$ with $m\geq 3$, the formula for
$\theta_X(\mu)$ is a signed sum of products of tomos. 
We now explain this product, which we will denote $\comp$. 

\begin{df}  \label{defin:composition}
Let $P, Q \in \parti{n}$ be two partitions and let $I \subseteq [n]$.  We say that $P$ 
and $Q$ are \emph{compatible through $I$} if $I \in P$ and $I^C \in Q$.

Write $P = \{I, I_1, \ldots, I_r \}$ and $Q = \{ I^C, J_1, \ldots, J_s  \}$.  
Then we define a new partition $P \comp Q : = \{ I_1, \ldots, I_r, J_1, \ldots, J_s \}$.
\end{df}

If $P$ and $Q$ are compatible, then there is a unique $I$ satisfying the 
definition, and so $P \comp Q$ is well-defined.

Let $P$ and $Q$ be compatible partitions.  Then we can define a product
$V_P \times V_Q \; \longrightarrow \; \; V_{P \comp Q}$, 
also denoted $\comp$, as follows. If $e_K \in E_K$ for 
$K = I_1, \ldots, I_r, J_1, \ldots, J_s$, then
\begin{equation}
(\phi \comp \psi) (e_{I_1}, \ldots, e_{I_r}, e_{J_1}, \ldots, e_{J_s}) : 
= \pair{\phi (e_{I_1}, \ldots, e_{I_r} )}{\psi( e_{J_1}, \ldots, e_{J_s})}
\end{equation}
where $\pair{}{}$ is the pairing of the bundles $E_{I^C}$ and $E_I$, which are 
dual to each other.

\begin{ej}
Consider $\phi\comp\psi$ where $\phi = (1,2,03)$ and $\psi = (0,3,12)$
as in Example \ref{ej:X3}. Here $I = \{0,3\}$. Writing $P$ and $Q$ for the
partitions, we have $P\comp Q = \{1,2,3,0\}$. Now
$$
(\phi\comp\psi)(e_1, e_2, e_3, e_0) = \pair{\phi(e_1,e_2)}{\psi(e_3, e_0)},
$$
where the pairing is $E_{12}\times_{M} E_{30}\to\R.$
\end{ej}

We can finally describe $\theta_{X_k}^{\BE}$.

\begin{lem} \label{lem:thetaX}
Let $\phi = (\phi_P)_{P \in \parti{n}}$ and let 
$\psi = (\psi_P)_{P \in \parti{n}}: = \theta_{X_k}^{\BE}(\phi)$.  Then
\begin{equation} \label{eq:Xkaction1}
\psi_P  = \phi_P
\end{equation}
if $0$ and $k$ are together in $P$, and
\begin{equation} \label{eq:Xkaction2}
\psi_P = 
\sum_{j=1}^{n-1}  (-1)^j \sum_{P = Q_1 \comp \ldots \comp Q_j}  \phi_{Q_1} \comp \ldots \comp \phi_{Q_j}
\end{equation}
otherwise.  The last sum in \eqref{eq:Xkaction2} is taken over all $j$-tuples of 
partitions $Q_1, \ldots, Q_j \in \parti{n}$ such that $0$ and $k$ are separate 
in $Q_i$ for all $i$;  $Q_i$ and $Q_{i+1}$ are compatible through $I_i$ for all 
$i = 1, \ldots, j-1$;  $I_i \neq I_{i+1}^C$ (otherwise the iterated composition 
does not make sense); and $P = Q_1 \comp \ldots \comp Q_j$.  
\end{lem}

\begin{proof}
Consider Figure \ref{fig:commdiag} with $W = X_k$. 
Let $A : = E_{\hop{n} \setminus k}$.  
Recall that $\decom{E}$ and 
$\decom{E^{X_k}}$ are dual vector bundles over $\decom{A}$.  
Let $e \in \decom{E}$ and let $e' \in \decom{E^{X_k}}$ be elements on fibers over 
the same $a \in A$.  Then 
\begin{equation} \label{eq:pairing}
\pair{e}{e'} = \pair{\Omega^0(\phi)(e)}{\Omega^k(\psi)(e')}.
\end{equation}
where $\pair{}{}$ is the pairing of the dual bundles over $\decom{A}$.   
Now we substitute \eqref{eq:defgk} and \eqref{eq:stato} in \eqref{eq:pairing} 
and we obtain \eqref{eq:Xkaction1} and \eqref{eq:Xkaction2}.
\end{proof}

\begin{ej}
In the case $n=4$, $k=1$, a typical action on a 3-tomo is 
$$
\theta_{X_1}(1,2,3,04) = -(1,2,3,04) + (04,2,13)\comp(1,3,024) + (04,3,12)\comp(1,2,034). 
$$
For the unique 4-tomo, $\theta_{X_1}(0,1,2,3,4)$ is given by 
\begin{multline*} 
- (1,2,3,4,0) + (0,2,134)\comp(1,3,4,02) + (0,3,124)\comp(1,2,4,03) + (0,4,123)\comp(1,2,3,04) \\
+ (14,2,3,0)\comp(1,4,023) - (0,2,134)\comp(14,3,02)\comp(1,4,023) - (0,3,124)\comp(14,2,03)\comp(1,4,023)\\
+ (13,2,4,0)\comp(1,3,024) - (0,2,134)\comp(13,4,02)\comp(1,3,024) - (0,4,123)\comp(13,2,04)\comp(1,3,024)\\ 
+ (12,3,4,0)\comp(1,2,034) - (0,3,124)\comp(12,4,03)\comp(1,2,034) - (0,4,123)\comp(12,3,04)\comp(1,2,034)
\end{multline*}
\end{ej}

Using Lemma \ref{lem:thetaX} and Equation \eqref{eq:quasiaction}, it is in principle
possible to obtain the explicit form of $\theta^{\BE}_W$ for any $W \in \dg_n$.  
This is a computation of intimidating length, which we will not attempt here. 
It will be sufficient to consider certain special cases. 

Let $\phi$ be a 2-tomo. Since there are no factorizations of the partition, 
$\theta_{X_k}(\phi)$, for each $X_k$ and hence each $W\in\dg_n$, must be $\pm\phi'$ 
for another 2-tomo $\phi'$. For $W\in K_{n+1}$ the integers are unchanged. This
provides a direct proof of the following corollary of Lemma \ref{lem:thetaX}. 

\begin{cor}
For $W\in K_{n+1}$ and $\phi$ a $2$-tomo, $\theta_W(\phi) = \pm\phi$. 
\end{cor}

It follows that every nonidentity element of $K_{n+1}$ has order $2$, and so 
$K_{n+1}$ is abelian and a direct product of $C_2$s. 

Returning to the discussion following (\ref{eq:kernelwords}), since by
Corollary \ref{cor:faithfulK} the restriction 
of $\theta$ to $K_{n+1}$ is faithful, we can drop the notation $\theta$. 
As noted already, we have
$X_i^2 = 1$ from standard duality, and $(X_iX_j)^3 = 1$ from the double case. 

\begin{prop}
\label{prop:ijik}
$K_{n+1}$ is generated by the words $(X_iX_jX_iX_k)^2$, for $i, j,k$ distinct, 
each of which has order $2$, and their conjugates.
\end{prop}

In the triple case, the conjugates of $(X_iX_jX_iX_k)^2$ are of the same form;
for $n\geq 4$ this is no longer so. 

In the next section we obtain a precise combinatorial description of
the group $K_{n+1}$. 

\section{The group $K_{n+1}$}
\label{sect:tgk}

To calculate the order of $K_{n+1}$, we identify it with a set of graphs. 

Denote the set of subgraphs of the complete graph on $n+1$ vertices, 
labeled $0, 1, \ldots, n$, by  $\gra{n+1}$. Under symmetric difference,
$\gra{n+1}$ is a group. Every nonidentity element is of order 2, and
$\gra{n+1}$ is isomorphic to the direct product of $\binom{n+1}{2}$ copies 
of $C_2$. 

Denote by $\ed{i}{j}$ the element of $\gra{n+1}$ for which the only edge is
the edge joining vertices $i$ and $j$. The $\{ \ed{i}{j} \;\vert\; 0\leq i < j \leq n \}$
are a set of generators for $\gra{n+1}$. 

There is a natural action of $S_{n+1}$ on $\gra{n+1}$ by permuting the vertices, 
which makes $\gra{n+1}$ into a $S_{n+1}$-module. There is also an action of $S_{n+1}$
on $K_{n+1}$ arising from the exact sequence (\ref{eq:ses}); since $K_{n+1}$ is abelian, 
the conjugation action of $\dg_n$ quotients to an action of $S_{n+1}$. 

We define an action $\overline{\theta} \co \gra{n+1} \times \G{\BE} \longrightarrow \G{\BE}$.
It is enough to describe the action of the generators $\ed{i}{j}$ of $\gra{n+1}$.  
Let $\phi = (\phi_P)_{P \in \parti{n}}$ and let 
$\psi = (\psi_P)_{P \in \parti{n}} : = \overline{\theta}_{\ed{i}{j}} (\phi)$.  Then
$\psi_P = \phi_P$ if $i$ and $j$ are together in $P$, and $\psi_P = - \phi_P$ otherwise.

This gives us an identification of $K_{n+1}$ with a subgroup of $\gra{n+1}$, 
and an embedding of $\dg_n$ in a semidirect product. 

\begin{thm}  
\label{thm:final}
For every $n$, there is an injective group homomorphism  
$\Psi\co  \dg_n \to \gra{n+1} \rtimes S_{n+1}$
defined by
\begin{equation*}
\Psi (X_i) = ( \ed{0}{i}, (0\,i))
\end{equation*}
for  $i = 1, \ldots, n$.
\end{thm}

\begin{proof}
That $\Psi(X_i)^2 = ( \ed{0}{i}, (0\,i))^2$ is the identity is clear. For $i\neq j$
we have
$$  
(\ed{0}{i}, (0\,i))\,(\ed{0}{j}, (0\,i))\,(\ed{0}{i}, (0\,i))
= (\ed{0}{i} + \ed{i}{j} + \ed{j}{0}, (i\,j))
$$
and the square of this is again the identity. Lastly, for $i.j.k$ distinct we have, 
$$  
(\ed{0}{i},(0\,i))\,(\ed{0}{j},(0\,j))\,(\ed{0}{i},(0\,i))\,(\ed{0}{k},(0\,k))
= (\ed{k}{j} + \ed{i}{j} +\ed{k}{i} + \ed{0}{k}, (i\,j)(0\,k)). 
$$
Squaring this, the $S_{n+1}$ component vanishes and the $\gra{n+1}$ component is
$\ed{0}{i} + \ed{i}{k} +\ed{k}{j} + \ed{j}{0}$, which is of order two. 
Thus $\Psi$ is well-defined and defines
an injective homomorphism $\dg_n \to \gra{n+1} \rtimes S_{n+1}$. 
\end{proof}

Using $\Psi$, define $\psi \co K_{n+1} \to \gra{n+1}$ by
\begin{equation}
\label{df:iota}
\psi((X_iX_jX_iX_k)^2) \; = \; \ed{0}{i} + \ed{0}{j} + \ed{k}{i} + \ed{k}{j} \; = \;
\raisebox{1.2pc}{
\xymatrix@=8pt{0 \ar@{-}[r] \ar@{-}[d] & i \ar@{-}[d] \\ j \ar@{-}[r] & k}
}
\end{equation}

\begin{thm}  \label{thm:kernel}
The map $\psi$ is an injective morphism 
of $S_{n+1}$--modules, and commutes with the actions of 
$K_{n+1}$ and $\gra{n+1}$ on $\G{\BE}$. That is,  
\begin{equation}  \label{eq:iota}
\overline{\theta}_{\psi(W)} = \theta_{W}
\end{equation}
for all $W \in K_{n+1}$. The image $\psi(K_{n+1})$ is the subgroup of $\gra{n+1}$ 
consisting of graphs such that
\begin{itemize}
\item each vertex has even valency, and
\item the total number of edges is even.
\end{itemize}
\end{thm}

\begin{proof}
Recall that we are identifying the group $K_{n+1}$ with its action on $\G{\BE}$.  
Since $K_{n+1}$ is an abelian group in which every non-identity element 
has order 2, we only need check that $\psi$ is $S_{n+1}$--equivariant and 
that \eqref{eq:iota} holds. Injectivity follows from the faithfulness of $\theta$. 
\end{proof}

In the following, we will abuse notation and think of $K_{n+1}$ both as a subgroup of 
$\gra{n+1}$ and as a subgroup of $\dg_{n+1}$.  By counting the number of graphs 
satisfying the two conditions at the end of Theorem \ref{thm:kernel} we obtain

\begin{cor}
\label{cor:orderK}
As a group, $K_{n+1}$ is isomorphic to the direct product of $\frac{1}{2}(n+1)(n-2)$ 
copies of $C_2$.   In particular
$$
\left|  \dg_n \right|  = 2^{\frac{1}{2}(n+1)(n-2)} (n+1)!
$$
\end{cor}

\section{The kernel for $n = 4$}
\label{sect:K5}

To give an example in detail which displays all the features described in 
\S\ref{sect:tgk}, consider the case $n = 4$. The computations here were done 
in part by hand and in part using a Java program written by Ms.~ Diksha Rajen, 
a graduate student in Computer Science at Sheffield, under the supervision of 
Dr.~Mike Stannett. 

The group $K_5$ contains 12 elements of the form $(ijik)^2$, where we abbreviate each
$X_i$ to $i$. We give these arbitrary labels as follows:
\begin{gather*}
A:= (1213)^2,\ B:= (1312)^2,\ C:= (2321)^2,\ D:= (1214)^2,\ E:= (1412)^2,\ F:= (2421)^2,\\ 
K:= (1413)^2,\ L:= (1314)^2,\ M:= (4341)^2,\ P:= (4243)^2,\ Q:= (4342)^2,\ R:= (2324)^2. 
\end{gather*}
Of course $A = (2123)^2 = (3121)^2 = (3212)^2$ also, and likewise for the other elements. 
The action of these on the 25 2-tomos is given in Table~\ref{table:action} at the end of 
the paper. 

To express the products of these elements we introduce three more labels,
$$
T := AD,\quad U := BF,\quad V := AQ, 
$$
We also encounter products which coincide with one of these 15 elements on the 2-tomos of
the form $2+2+1$ but have reversed signs on the 2-tomos of the form $3+1+1$. We therefore
define an element $i\in K_5$ in terms of its action: $i$ preserves each 2-tomo of
the form $2+2+1$ and reverses the signs on the 2-tomos of the form $3+1+1$.

Write $a:= Ai,\ b := Bi,\ \dots, v:= Vi$ and finally write $I$ for the identity element of
$K_5$. This completes the description of the 32 elements of $K_5$ in terms of their action
on 2-tomos. The multiplication table is given in Table \ref{table:mult}. 

The elements of $\gra{5}$ corresponding to the 12 elements $A, \dots, R$ are of the
type shown in (\ref{df:iota}). Calculating symmetric differences, we obtain 
Figure~\ref{fig:TUV} for $T$, $U$ and $V$. 
\begin{figure}[h]
\begin{center}
\subfloat[$T$]%
{\raisebox{1.2pc}{
\xymatrix@=8pt{1 \ar@{-}[r] \ar@{-}[d] & 3 \ar@{-}[d] \\ 4 \ar@{-}[r] & 2}
}}
\qquad
\subfloat[$U$]%
{\raisebox{1.2pc}{
\xymatrix@=8pt{1 \ar@{-}[rr] \ar@{-}[rd] && 0 \ar@{-}[rd] \ar@{-}[ld] \ar@{-}[rr] && 3 \ar@{-}[ld]\\ 
& 4 && 2 &}
}}
\qquad
\subfloat[$V$]%
{\raisebox{1.2pc}{
\xymatrix@=8pt{1 \ar@{-}[rr] \ar@{-}[rd] && 0 \ar@{-}[rd] \ar@{-}[ld] \ar@{-}[rr] && 2 \ar@{-}[ld]\\ 
& 3 && 4 &}
}}
\end{center}
\caption{Elements of $\gra{5}$ corresponding to $T$, $U$ and $V$. \label{fig:TUV}}
\end{figure}
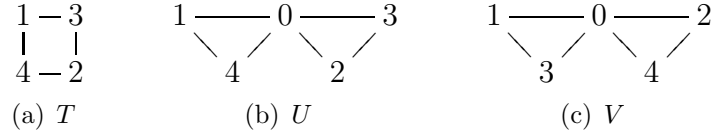

The graph for $i$ is the complete graph on the 5 vertices, and that for $I$ is the
graph with no edges. Multiplication by $i$ converts a graph to its complement. 

From Figure~\ref{fig:TUV}(a) we see that $T = 1Q1$, the conjugate of $Q$ by $X_1$. 
Likewise, it is equal to $2M2$, $3D3$ and $4A4$. In this way arbitrary conjugates can
be calculated. 

\section{Description of $\dg_n$}
\label{sect:description}

Given that $\dg_n$ is an extension \eqref{eq:ses} of $S_{n+1}$ by an abelian group, a 
natural question to ask is whether the extension splits; that is, whether $\dg_n$ is 
isomorphic to the semidirect product $K_{n+1} \rtimes S_{n+1}$.   We proved in 
\cite{Gracia-SazM:2009} that this is not the case for $n=3$.    

\begin{prop}
The extension \eqref{eq:ses} is split for $n = 4$.   
\end{prop}

\begin{proof}
Using the presentation in Proposition \ref{prop:dg4rels} below, GAP \cite{GAP449}
shows that $\dg_4$ has subgroups isomorphic to $S_5$. Let $S$ be any
such subgroup. Then if the restriction of $\dg_4\to S_5$ to $S\to S_5$ is not an 
isomorphism, it must have kernel $A_5$ or $S$ itself. But the kernel of $\dg_4\to S_5$ 
is a product of $C_2$s. 
\end{proof}

\begin{prop}  \label{prop:n2mod4}
If $n= 2 \pmod{4}$, then the extension \eqref{eq:ses} splits.  
Specifically, let $\gamma_{ij} \in K_{n+1}$ correspond to the complete graph on the 
vertices $[n] \setminus i, j$.    
Then there is an isomorphism  $\Phi\co  \dg_n \to K_{n+1} \rtimes S_{n+1}$ defined by
\begin{equation}
\Phi(X_j) = \left( \gamma_{0j}, (0\,j) \right)
\end{equation}
for $j = 1, \ldots, n$.
\end{prop}

\begin{proof}
First, we notice that for $\gamma_{0j}$ to be an element of $K_{n+1}$ it needs to 
have an even number of edges and only even-valency vertices.  This happens exactly 
when $n=2 \pmod{4}$. To check that $\Phi$ is a homomorphism, we need to check the 
three conditions:
$$
\Phi(X_i)^2 = 1, \qquad
\left(\Phi(X_i)\Phi(X_j) \right)^3 = 1, \qquad
\left( \Phi(X_i) \Phi(X_j) \Phi(X_i) \Phi(X_k) \right)^2  = 
\raisebox{1.2pc}{
\xymatrix@=8pt{0 \ar@{-}[r] \ar@{-}[d] & i \ar@{-}[d] \\ j \ar@{-}[r] & k}
},
$$
for all distinct $i, j, k$.  The first is straightforward. For the second we have
\begin{equation}
\label{eq:gamma}
\left(\gamma_{0i}, (0\,i)\right)\,\left(\gamma_{0j}, (0\,j)\right)\,\left(\gamma_{0i}, (0\,i)\right) 
= \left(\gamma_{0i} + \gamma_{ij} + \gamma_{0j},\ (i\,j)\right). 
\end{equation}
Squaring this gives the identity, since $\gamma_{0i} + \gamma_{ij} + \gamma_{0j}$ is
preserved by $(i\,j)$. 

Now write $\beta_{i,0,j} = \gamma_{0i} + \gamma_{ij} + \gamma_{0j}$. This consists
of the null graph on $\{i,0,j\}$ and the full graph on the complementary set 
of vertices $[n]\backslash\{0,i,j\}$, with each of $0,i,j$ joined to each of the
vertices in $[n]\backslash\{0,i,j\}$. We have
$$
\left(\gamma_{0i}, (0\,i)\right)\,\left(\gamma_{0j}, (0\,j)\right)\,\left(\gamma_{0i}, (0\,i)\right)\,\left(\gamma_{0k}, (0\,k)\right) 
= (\beta_{i,0,j} + \gamma_{0k},\ (i\,j)\,(0\,k)). 
$$
Squaring this, the kernel term is $\beta_{i,0,j} + \beta_{i,j,k}$ and this is the
graph in (\ref{df:iota}). 
\end{proof}

We do not know whether \eqref{eq:ses} splits for $n\geq 8$ a multiple of $4$, or what the
situation is for odd values $\geq 5$. In particular we withdraw the announcement at the 
end of Section~4 of \cite{Gracia-SazM:2009}, that \eqref{eq:ses} splits if and only 
if $n$ is even. 

Although $\dg_n$ is not always a semidirect product, we can always see it as a subgroup of 
a semidirect product, as Theorem \ref{thm:final} demonstrates. Indeed
Theorem \ref{thm:final} provides perhaps the most enlightening model for $\dg_n$.  
For a dualization operation $W \in \dg_n$ with $\Psi(W) = (\gamma, \lambda)$,
the permutation $\lambda \in S_{n+1}$ gives the action of the functor $W$ 
on the building bundles, and the graph $\gamma \in \gra{n+1}$ 
tells us the action of $W$ on the ``set of changes of decompositions'' $\G{\BE}$.  
\label{pref:flip}
One may regard $\gamma$ as measuring to what extent $W$ fails to be merely a rearrangement
of the building bundles. 

Equivalently, if $\Psi(W_1) = (\gamma_1, \lambda_1)$ and 
$\Psi(W_2) = (\gamma_2, \lambda_2)$ have $\lambda_1 = \lambda_2$, then 
$\gamma_1 - \gamma_2$ measures the failure of $W_1$ and $W_2$ to be naturally isomorphic 
functors.

As a final general result, we compute the centre of $\dg_n$.   
Let $n \geq 2$.  A central element needs 
to be in $K_{n+1}$, since $S_{n+1}$ has trivial centre, and it needs to be invariant under 
the $S_{n+1}$--action.  The only options are the empty graph, which is the identity, and 
the complete graph on $n+1$ vertices.  The complete graph is only an element of 
$K_{n+1}$ when $n=0 \pmod{4}$.  Hence we conclude:

\begin{prop}
If $n$ is a multiple of $4$, then $\vert Z(\dg_n) \vert = 2$.  
Otherwise, $\dg_n$ has trivial centre.
\end{prop}

A remaining question is to describe $\dg_n$ in terms of relations in the $X_i$. From
Proposition \ref{prop:ijik} we know that in $\dg_n$ for any $n\geq 3$ the relations
\begin{equation}
\label{eq:rels}
X_i^2,\ (X_iX_j)^3,\ (X_iX_jX_iX_k)^4,
\end{equation}
where $i, j, k$ are distinct, hold. For $n = 4$ it is easy to verify with GAP 
\cite{GAP449} that the group defined by these relations is infinite (and has Mathieu
groups as quotients). From Table~\ref{table:mult} we see that $AK = P$ and that
$AD = MQ$. These give the relations 
\begin{equation}
\label{eq:rels4}
12131213.1413141.4243424, \quad
12131213.12141214.24342434.14341434 
\end{equation}
We would like to have an interpretation of these relations comparable to the interpretations
of $(ij)^3 = 1$ in terms of the duality of doubles, or the notion of cornering in 
\cite{Mackenzie:2005dts}.

\begin{prop}
\label{prop:dg4rels}
The group $\dg_4$ is the group on generators $X_i$, $1\leq i \leq 4$, subject to the
relations {\rm (\ref{eq:rels})} and {\rm (\ref{eq:rels4})}. 
\end{prop}

This is again a straightforward calculation in GAP \cite{GAP449}. The general case 
is the subject of ongoing work. 

\subsection*{Acknowledgements}

Gracia-Saz's research was partially supported by fellowships from
the Secretar\'ia de Estado de Universidades e Investigacion del Ministerio
Espa\~nol de Educaci\'on y Ciencia and from the Japanese Society for the Promotion of
Science. Both authors thank the University of Sheffield for funding visits by
Gracia-Saz to Sheffield. 
In preparing the paper, Mackenzie benefitted from a Java program written by 
Ms. Diksha Rajen, 
a graduate student in Computer Science at Sheffield, under the supervision
of Dr. Mike Stannett; he is grateful to them both. 

\newcolumntype{Z}{>{$}c<{$}}
\newcolumntype{L}{>{$}l<{$}}
\newcolumntype{R}{>{$}r<{$}}

\begin{landscape}
\begin{table}
\begin{tabular}{|L|ZZZ|ZZZ|ZZZ|ZZZ|ZZZ|}
\hline
            & A    &   B   & C  & D   & E   & F  & K   & L   & M   & P   & Q   & R & T & U & V \\
\hline
 (12,0,34)  & +    &   -   & -  & +   & -   & -  & -   & -   & +   & -   & +   & - & + & + & + \\
 (13,0,24)  & -    &   +   & -  & -   & -   & +  & -   & +   & -   & +   & -   & - & + & + & + \\
 (14,0,23)  & -    &   -   & +  & -   & +   & -  & +   & -   & -   & -   & -   & + & + & + & + \\
&&&&&&&&&&&&&&&\\                                                                         
 (02,1,34)  & -    &   +   & -  & -   & +   & -  & -   & -   & +   & +   & +   & + & + & - & - \\
 (03,1,24)  & +    &   -   & -  & -   & -   & +  & +   & -   & -   & +   & +   & + & - & - & + \\
 (04,1,23)  & -    &   -   & +  & +   & -   & -  & -   & +   & -   & +   & +   & + & - & + & - \\
&&&&&&&&&&&&&&&\\                                                                                 
 (01,2,34)  & -    &   -   & +  & -   & -   & +  & +   & +   & +   & -   & +   & - & + & - & - \\
 (03,14,2)  & +    &   -   & -  & -   & +   & -  & +   & +   & +   & +   & -   & - & - & + & - \\
 (04,13,2)  & -    &   +   & -  & +   & -   & -  & +   & +   & +   & -   & -   & + & - & - & + \\
&&&&&&&&&&&&&&&\\                                                                         
 (01,24,3)  & -    &   -   & +  & +   & +   & +  & -   & -   & +   & +   & -   & - & - & - & + \\
 (02,14,3)  & -    &   +   & -  & +   & +   & +  & +   & -   & -   & -   & +   & - & - & + & - \\
 (04,12,3)  & +    &   -   & -  & +   & +   & +  & -   & +   & -   & -   & -   & + & + & - & - \\
&&&&&&&&&&&&&&&\\                                                                         
 (01,23,4)  & +    &   +   & +  & -   & -   & +  & -   & -   & +   & -   & -   & + & - & + & - \\
 (02,13,4)  & +    &   +   & +  & -   & +   & -  & -   & +   & -   & -   & +   & - & - & - & + \\
 (03,12,4)  & +    &   +   & +  & +   & -   & -  & +   & -   & -   & +   & -   & - & + & - & - \\
&&&&&&&&&&&&&&&\\                                                                                
\hline\hline
 (012,3,4)  & +    &   +   & +  & +   & +   & +  & -   & -   & +   & -   & +   & - & + & + & + \\
 (013,2,4)  & +    &   +   & +  & -   & -   & +  & +   & +   & +   & +   & -   & - & - & + & - \\
 (014,2,3)  & -    &   -   & +  & +   & +   & +  & +   & +   & +   & -   & -   & + & - & - & + \\
\hline                                                                                  
 (023,1,4)  & +    &   +   & +  & -   & +   & -  & +   & -   & -   & +   & +   & + & - & - & + \\
 (024,1,3)  & -    &   +   & -  & +   & +   & +  & -   & +   & -   & +   & +   & + & - & + & - \\
 (034,1,2)  & +    &   -   & -  & +   & -   & -  & +   & +   & +   & +   & +   & + & + & + & + \\
\hline                                                                                  
 (123,0,4)  & +    &   +   & +  & +   & -   & -  & -   & +   & -   & -   & -   & + & + & - & - \\
 (124,0,3)  & +    &   -   & -  & +   & +   & +  & +   & -   & -   & +   & -   & - & + & - & - \\
 (134,0,2)  & -    &   +   & -  & -   & +   & -  & +   & +   & +   & -   & +   & - & + & - & - \\
 (234,0,1)  & -    &   -   & +  & -   & -   & +  & -   & -   & +   & +   & +   & + & + & - & - \\
\hline                                                                                  
            & A    &   B   & C  & D   & E   & F  & K   & L   & M   & P   & Q   & R & T & U & V\\
\hline
\end{tabular}
\caption{\label{table:action}See \S\ref{sect:K5}}
\end{table}
\end{landscape}

\begin{table}
\begin{tabular}{||Z||Z|Z||Z|Z|Z||Z|Z|Z||Z|Z|Z||Z|Z|Z||}
\hline
  & B   & C   & D & E       & F & K & L & M & P & Q    & R & T & U & V\\
\hline\hline
A & C   & B   & T & \ell  & r & P & e & U & K & V      & f & D & M & Q \\  
B &     & A   & k & Q       & U & d & v & r & t & E    & m & p & F & \ell \\
C &     &     & p & V       & M & t & q & F & d & \ell & u & k & r & E \\
\hline\hline
D &     &     &   & F       & E & b & R & V & c & U    & L & A & Q & M \\
E &     &     &   &         & D & u & a & p & m & B    & t & r & k & C \\
F &     &     &   &         &   & q & t & C & v & k    & a & \ell & B & p \\
\hline\hline
K &     &     &   &         &   &   & M & L & A & f    & V & c & e & R \\
L &     &     &   &         &   &   &   & K & U & c    & D & f & P & b \\
M &     &     &   &         &   &   &   &   & e & T    & b & Q & A & D \\
\hline\hline
P &     &     &   &         &   &   &   &   &   & R    & Q & b & L & f \\
Q &     &     &   &         &   &   &   &   &   &      & P & M & D & A \\
R &     &     &   &         &   &   &   &   &   &      &   & e & c & K \\
\hline\hline
T &   &   &   &   &   &   &   &   &   &   &   &  & V & U \\
U &   &   &   &   &   &   &   &   &   &   &   &  &   & T \\
\hline\hline
\end{tabular}
\caption{\label{table:mult}See \S\ref{sect:K5}}
\end{table}

\newcommand{\noopsort}[1]{} \newcommand{\singleletter}[1]{#1} \def\cprime{$'$}
  \def\cprime{$'$}


\begin{thebibliography}{10}

\bibitem{Besse:MAWGC}
A.~L. Besse.
\newblock {\em Manifolds all of whose geodesics are closed}, volume~93 of {\em
  Ergebnisse der Mathematik und ihrer Grenzgebiete [Results in Mathematics and
  Related Areas]}.
\newblock Springer-Verlag, Berlin, 1978.
\newblock With appendices by D. B. A. Epstein, J.-P. Bourguignon, L.
  B{\'e}rard-Bergery, M. Berger and J. L. Kazdan.

\bibitem{Dieudonne:III}
J.~Dieudonn{\'e}.
\newblock {\em Treatise on analysis. {V}ol. {III}}.
\newblock Academic Press, New York, 1972.
\newblock Translated from the French by I. G. MacDonald, Pure and Applied
  Mathematics, Vol. 10-III.

\bibitem{GAP449}
The GAP~Group.
\newblock {\em {GAP -- Groups, Algorithms, and Programming, Version 4.4.9}},
  2006.
\newblock \url{http://www.gap-system.org}.

\bibitem{GrabowskiR:2009}
J.~Grabowski and M.~Rotkiewicz.
\newblock Higher vector bundles and multi-graded symplectic manifolds.
\newblock {\em J. Geom. Phys.}, 59(9):1285--1305, 2009.

\bibitem{Gracia-SazM:2009}
A.~Gracia-Saz and K.~C.~H. Mackenzie.
\newblock Duality functors for triple vector bundles.
\newblock {\em Lett. Math. Phys.}, 90(1-3):175--200, 2009.

\bibitem{Knuth:1}
D.~E. Knuth.
\newblock {\em The art of computer programming}.
\newblock Addison-Wesley Publishing Co., Reading, Mass.-London-Amsterdam, third
  edition, 1997.
\newblock Volume 1: Fundamental algorithms, Addison-Wesley Series in Computer
  Science and Information Processing.

\bibitem{Mackenzie:1999}
K.~C.~H. Mackenzie.
\newblock On symplectic double groupoids and the duality of {P}oisson
  groupoids.
\newblock {\em Internat. J. Math.}, 10(4):435--456, 1999.

\bibitem{Mackenzie:2005dts}
K.~C.~H. Mackenzie.
\newblock Duality and triple structures.
\newblock In {\em The breadth of symplectic and {P}oisson geometry}, volume 232
  of {\em Progr. Math.}, pages 455--481. Birkh\"auser Boston, Boston, MA, 2005.

\bibitem{Mackenzie:GT}
K.~C.~H. Mackenzie.
\newblock {\em General theory of {L}ie groupoids and {L}ie algebroids}, volume
  213 of {\em London Mathematical Society Lecture Note Series}.
\newblock Cambridge University Press, Cambridge, 2005.

\bibitem{Mackenzie:2011}
K.~C.~H. Mackenzie.
\newblock Ehresmann doubles and {D}rinfel'd doubles for {L}ie algebroids
              and {L}ie bialgebroids.
\newblock {\em J. Reine Angew. Math.}, {658}:193--245, 2011. 

\bibitem{MackenzieX:1994}
K.~C.~H. Mackenzie and P.~Xu.
\newblock Lie bialgebroids and {P}oisson groupoids.
\newblock {\em Duke Math. J.}, 73(2):415--452, 1994.

\bibitem{Pradines:DVB}
J.~Pradines.
\newblock Fibr{\'e}s vectoriels doubles et calcul des jets non holonomes.
\newblock Notes polycopi{\'e}es, Amiens, 1974.

\bibitem{Pradines:1988}
J.~Pradines.
\newblock Remarque sur le groupo\"\i de cotangent de {W}einstein-{D}azord.
\newblock {\em C. R. Acad. Sci. Paris S\'er. I Math.}, 306(13):557--560, 1988.

\bibitem{Tulczyjew:1977}
W.~M. Tulczyjew.
\newblock A symplectic formulation of particle dynamics.
\newblock In {\em Differential geometric methods in mathematical physics
  ({P}roc. {S}ympos., {U}niv. {B}onn, {B}onn, 1975)}, pages 457--463. Lecture
  Notes in Math., Vol. 570. Springer, Berlin, 1977.

\bibitem{Voronov:QmMt}
Th.~Th. Voronov.
\newblock {$Q$-manifolds and Mackenzie theory}.
\newblock {\em Comm. Math. Phys.}
\newblock (to appear).

\end{thebibliography}
\end{document}